\documentclass[11pt]{article}

\usepackage{amssymb}
\usepackage{amsmath,amsfonts,amssymb,amsthm,stmaryrd} 
\usepackage{mathrsfs,tcolorbox,hyperref} 
\usepackage{graphicx,subcaption,caption} 
\usepackage{fancyhdr}
\usepackage{mathbbol}
\usepackage{tcolorbox}
\usepackage[export]{adjustbox}
\usepackage{multirow}
\newtheorem{theorem}{Theorem}[section]
\newtheorem{lemma}{Lemma}[section]
\newtheorem{corollary}{Corollary}[section]
\newtheorem{remark}{Remark}[section]
\newtheorem{definition}{Definition}[section]

\newcommand{\Dxx}{\mathcal{D}_{xx}} \newcommand{\Dyy}{\mathcal{D}_{yy}}
\newcommand{\Dx}{\mathcal{D}_{x}} \newcommand{\Dy}{\mathcal{D}_{y}}
\newcommand{\Dxy}{\mathcal{D}_{xy}}
\newcommand{\Hx}{\mathcal{H}_x} \newcommand{\Hy}{\mathcal{H}_y}
\newcommand{\HH}{\mathcal{H}}

\newcommand{\MM}{\mathcal{M}_\Gamma}
\newcommand{\Bx}{\mathcal{B}_x}
\newcommand{\By}{\mathcal{B}_y} \newcommand{\Bg}{\mathcal{B}_\Gamma}
\newcommand{\Ax}{\mathcal{A}_x}
\newcommand{\Ay}{\mathcal{A}_y}
\newcommand{\ZZ}{\mathcal{Z}}
\newcommand{\UU}{\mathcal{U}}
\newcommand{\LL}{\mathcal{L}}
\newcommand{\Puvg}{\mathcal{P}_{u2v}^g} \newcommand{\Puvb}{\mathcal{P}_{u2v}^b}
\newcommand{\Pvug}{\mathcal{P}_{v2u}^g} \newcommand{\Pvub}{\mathcal{P}_{v2u}^b}

\newcommand{\uu}{\pmb{u}} \newcommand{\vv}{\pmb{v}} \newcommand{\ww}{\pmb{w}}
\newcommand{\nvec}{\pmb{n}} 
 
\newcommand{\ff}{\pmb{f}} 
\newcommand{\vvarphi}{\pmb{\varphi}}
\newcommand{\xx}{\pmb{x}} 
\newcommand{\ee}{\pmb{e}}
\newcommand{\pphi}{\pmb{\varphi}}
\newcommand{\ttau}{\pmb{\tau}}
\newcommand{\ssigma}{\pmb{\sigma}}
\newcommand{\eepsilon}{\pmb{\varepsilon}}

\newcommand{\triangulation}{\mathscr{K}_h}
\newcommand{\testspace}{\mathscr{V}_h^q}
\newcommand{\edges}{\mathscr{E}_h}\newcommand{\boundaryedges}{\mathscr{E}_h^\partial}
\newcommand{\interioredges}{\mathscr{E}_h^I}
\newcommand{\fdgrid}{\mathbf{x}_u}
\newcommand{\dggrid}{\mathscr{T}_v}
\newcommand{\gluegrid}{\mathscr{T}_g}

\newcommand{\Tu}{\mathcal{T}_u} \newcommand{\Tv}{\mathcal{T}_v}
\newcommand{\Tx}{\mathcal{T}_x}\newcommand{\Ty}{\mathcal{T}_y}

\newcommand\revone[1]{{\color{black} #1}}
\newcommand\revthree[1]{{\color{black} #1}}

\begin{document}

\title{A hybrid numerical method for elastic wave propagation in discontinuous media with complex geometry }

\author{Andreas Granath\thanks{Department of Mathematics and Mathematical Statistics, {Ume\aa} University, Sweden (andreas.granath@umu.se)}\quad and\quad Siyang Wang\thanks{Corresponding author. Department of Mathematics and Mathematical Statistics, {Ume\aa} University, Sweden (siyang.wang@umu.se)}}
\maketitle

\begin{abstract}
  We develop a high order accurate numerical method for solving the elastic wave equation in second-order form. We hybridize the computationally efficient Cartesian grid formulation of finite differences with geometrically flexible discontinuous Galerkin methods on unstructured grids by a penalty based technique. At the interface between the two methods, we construct projection operators for the pointwise finite difference solutions and discontinuous Galerkin solutions based on piecewise polynomials. In addition, we optimize the projection operators for both accuracy and spectrum. We prove that the overall discretization conserves a discrete  energy, and verify optimal convergence in numerical experiments. 
\end{abstract}



\section{Introduction}\label{sec: Introduction}

The elastic wave equation sees a lot of applications in the modeling of seismic waves occurring due to earthquakes and seismic events. Realistic problems often involve discontinuous material properties as a result of earth layers or complex geometries, making it impossible to obtain analytical solutions. It is therefore of great importance  to develop numerical schemes in order to obtain approximate solutions. Given the time dependency of the problem, it becomes crucial that the numerical scheme offers stability and sufficient accuracy. 

From the classical dispersion analysis by Kreiss and Oliger \cite{Kreiss3}  and Hagstrom and Hagstrom \cite{Hagstrom}, we know that high order methods are more efficient for wave dominated problems with sufficiently smooth solutions. This has led to a rich development of high order schemes for various wave problems. 

A method which has seen a lot of applications in the context of high order solvers for wave-dominated problems is the finite difference (FD) method. Finite differences are conceptually easy to understand, computationally efficient and straightforward to implement. However, it was challenging to derive a stable and high order finite difference discretization for initial boundary value problems. This challenge has largely been overcome by using finite difference operators satisfying the \textit{summation-by-parts} (SBP) property \cite{Kreiss1}. Operators with this property satisfy a discrete analogue of the integration-by-parts identity, which is the key to derive continuous energy estimates. This method is often successfully combined with the \textit{simultaneous-approximation-term} (SAT) method to implement boundary conditions via penalization \cite{Carpenter}, resulting in an SBP-SAT scheme. Alternatively, boundary conditions can be imposed strongly by the projection method  \cite{Olsson1,Olsson2} or the ghost point  method \cite{Sjogreen}. A drawback of using FD operators however is that they are inherently defined on rectangular grids. To resolve more general geometries, a curvilinear grid is often constructed using a coordinate transformation technique  \cite{Nordstrom1,Svard}. In certain cases, we need to partition the domain into subdomains that can be more easily mapped to a reference domain, resulting in a multi-block finite difference method \cite{Mattsson2}. The elastic wave equation has been successfully discretized with the multiblock SBP finite difference method, see e.g. \cite{Almquist1,Duru1} for the SAT method and \cite{Zhang,ZhangWangPetersson} for the ghost point method. The relation between the SAT method and the ghost point method was discussed in \cite{WangPetersson}. \revthree{It should be noted that the aforementioned papers considers the equation in second order form as this is known to, in general, result in fewer unknowns than first order formulations and improved order of convergence for SBP-SAT discretizations.}

Methods that allow unstructured meshes can be used to effectively resolve complex geometry. Most prominent of these methods are the finite element method (FEM), which allows a flexible triangulation of the computational domain. The continuous Galerkin (CG) method requires solving a linear system at each time step, even with an explicit time integrator. On the other hand, the discontinuous Galerkin (DG) method has a block diagonal mass matrix, and its inverse can be computed blockwise.   The DG method has been applied to the elastic wave equation in quite general settings, see e.g. \cite{Antonietti, Basabe}.

To leverage the computational efficiency of the FD method and the geometric flexibility of the DG method, a hybrid FD-DG method has been proposed. At the FD-DG interface, the discretizations have different induced norms, making it challenging to obtain a stable hybrid method with high order accuracy. This challenge was overcome by Kozdon and Wilcox in \cite{Kozdon}, where an FD-DG discretization was constructed for the wave equation in first-order form with optimal convergence. It was later shown in the setting of non-conforming FD-FD couplings that in order to obtain optimal convergence when considering problems in second-order form, one needs to use pairs of projection operators \cite{Almquist2}. In a recent paper, the work of Kozdon and Wilcox was combined with the notion of projection pairs to solve the wave equation in second-order form, obtaining optimal accuracy \cite{Wang1}. FD-FEM schemes have also been constructed previously for Maxwell's equation \cite{Beilina2013} and the elastic wave equation \cite{Gao}. In \cite{Lisitsa}, instead, a first-order velocity-stress formulation was discretized by DG coupled with a staggered grid FD formulation through a finite volume transition zone. 

In this work, we follow the coupling methodology in \cite{Wang1} and consider the elastic wave equation in second-order form in a bounded, isotropic and possibly discontinuous material. 
The first contribution of this paper is an FD-DG discretization of the elastic wave equation in second-order displacement form. We combine an SBP FD discretization with an interior penalty DG (IPDG) discretization using norm compatible projection pairs. It is shown to converge optimally and proven to satisfy a discrete energy estimate under appropriate choice of penalty parameters, guaranteeing stability. 

The second contribution is in the optimization of the projection operators for accuracy and spectrum, and its impact on  stability. When constructing the projection operators, accuracy conditions are used together with the norm compatibility to set up a system of equations for the coefficients in the projection operators. This system is in general underdetermined and results in a large set of free parameters. We demonstrate that  choosing these parameters without optimization, e.g. setting them to zero, yields a stiff system and instead investigate two different strategies of determining them. More specifically, we choose them to minimize the leading order $l^2$-error as well as two different measures of spectrum; the \revone{Frobenius} norm and an eigenvalue-flattening norm inspired by \cite{Kozdon}. This extends  the work of \cite{Wang1},  where only optimization of accuracy was considered when constructing the operators. The order in which we optimize accuracy and spectrum is also considered. Finally, we demonstrate that the discretization is not sensitive to the distribution of FD and DG nodes along the interface.

The rest of the paper is structured as follows. In Section~\ref{sec: the elastic wave equation}, we introduce the elastic wave equation together with the notion of FD operators satisfying the SBP property and the discontinuous Galerkin method. Then in Section~\ref{sec: the coupled scheme}, we begin by giving a detailed account of the projection operators and their properties, and then prove that the FD-DG discretization satisfies an energy estimate. We then demonstrate the accuracy and stability of the discretization by a series of numerical experiments in Section~\ref{sec: numerical results} and end with a brief discussion and summary in Section~\ref{sec: conclusion}.

\section{The elastic wave equation}

\label{sec: the elastic wave equation}
We consider a bounded domain $\Omega$ that can be partitioned into two subdomains $\Omega_u$ and $\Omega_v$ made up of isotropic materials and with contact along a horizontal interface $\Gamma$. Then the elastic wave equation can be written in component form in terms of the two displacement fields $\uu=[u_1,u_2]^T$ and $\vv=[v_1,v_2]^T$ as 
\begin{align}
    \rho_u\partial_t^2u_1&=\mu_u\Delta u_1+(\mu_u+\lambda_u)\partial_x(\partial_xu_1+\partial_yu_2),\nonumber\\
    \rho_u\partial_t^2u_2&=\mu_u\Delta u_2+(\mu_u+\lambda_u)\partial_y(\partial_xu_1+\partial_yu_2)\quad\text{ for }(x,y,t)\in\Omega_u\times(0,T],\label{eq:ContinuousFormulationU}\\
    \rho_v\partial_{t}^2v_1&=\mu_v\Delta v_1+(\mu_v+\lambda_v)\partial_x(\partial_xv_1+\partial_yv_2),\nonumber\\
    \rho_v\partial_{t}^2v_2&=\mu_v\Delta v_2+(\mu_v+\lambda_v)\partial_y(\partial_xv_1+\partial_yv_2)\quad\text{ for }(x,y,t)\in\Omega_v\times(0,T],\label{eq:ContinuousFormulationV}
\end{align}
where $\mu_u,\lambda_u,\mu_v,\lambda_v$ denotes the Lamé parameters of the respective domains, which are assumed to be constant, as well as the respective densities $\rho_u,\rho_v$ which also are constant. For a more detailed description of the above formulation, see e.g. \cite{Virta}. In order for the problem  to be well-posed, we must also impose initial conditions, boundary conditions and suitable interface conditions. To formulate the boundary conditions, we introduce the \revone{stress tensor} for an arbitrary smooth field $\uu=[u_1,u_2]^T$ as $\sigma(\uu)$ 

\begin{equation}
    \sigma(\uu)=\begin{bmatrix}
        (2\mu_u+\lambda_u)\partial_xu_1+\lambda_u\partial_yu_2 & \mu_u(\partial_yu_1+\partial_xu_2) \\
        \mu_u(\partial_yu_1+\partial_xu_2) & (2\mu_u+\lambda_u)\partial_yu_2+\lambda_u\partial_xu_1
    \end{bmatrix}.
\end{equation}

Moreover, let $\nvec$ denote the unit outward normal to $\partial\Omega/\Gamma$. and introduce the tractions $\ttau(\uu)=\sigma(\uu)\cdot\nvec$ and $\ttau(\vv)=\sigma(\vv)\cdot\nvec$. Then we formulate the traction free boundary conditions as
\begin{align*}
    \ttau(\uu)&=\pmb{0}\quad\text{ along }\partial\Omega_u/\Gamma,\\
    \ttau(\vv)&=\pmb{0}\quad\text{ along }\partial\Omega_v/\Gamma.
\end{align*} 

To ensure continuity of the displacement and traction along the interface, we  impose the following interface conditions
\begin{align}
    \uu&=\vv \qquad\text{ along }\Gamma,\label{eqn: contcond}\\
    \sigma(\vv)\cdot\nvec^++\sigma(\uu)\cdot\nvec^-&=\pmb{0}\qquad\text{ along }\Gamma \label{eqn: traccond},
\end{align}
where $\nvec^+=[0,1]^T$, $\nvec^-=[0,-1]^T$ .

Finally, it should be noted for completeness that the continuous formulation of \eqref{eq:ContinuousFormulationU}-\eqref{eq:ContinuousFormulationV} can readily be shown to satisfy an energy estimate, see e.g. \cite{Appelo}. We begin in Section~\ref{subsec: SBP formulation} by introducing the necessary background to discretize \eqref{eq:ContinuousFormulationU}-\eqref{eq:ContinuousFormulationV} in the SBP-SAT framework, and then in Section~\ref{subsec: dG formulation} we discretize it in the IPDG framework.
\subsection{SBP finite difference formulation}

\label{subsec: SBP formulation}
Let $I=[a,b]$ be an interval, $\xx=[x_0,...,x_N]$ a uniform partition of $I$ with $x_0=a,x_N=b$ and $x_j=a+jh$ where $h$ is the grid spacing. Furthermore, let $f,g$ be functions of sufficient regularity, then $\ff=[f(x_0),..,f(x_N)]^T$ and $\pmb{g}=[g(x_0),..,g(x_N)]^T$ shall denote their grid evaluations. Moreover, let $H=H^T$ be a quadrature matrix which is positive definite such that it approximates $\ff^TH\pmb{g}\approx\int_If(x)g(x)dx$ and induces a norm $\|\ff\|_H^2=\ff^TH\ff$. 

Next, we let $D_1\approx\partial_x$ and $D_2\approx\partial_x^2$ be narrow stencil FD approximations of the first and second derivatives with interior error $\mathcal{O}(h^{2p})$ and $\mathcal{O}(h^p)$ at a few points along the boundary due to the one-sided stencil. We then say that the operator $D_2$ is of order $2p$. The SBP property for the differential operators are then defined as follows \cite{Mattsson1}.
\begin{definition}
  The matrices $D_1$ and $D_2$  are said to satisfy the summation-by-parts (SBP) property if it holds
    $$HD_1+(HD_1)^T=B,$$
    $$D_2=H^{-1}(-A+BS),$$
    where $B=B_n-B_0$ with $B_n=\operatorname{diag}(0,\cdot\cdot\cdot,1), B_0=\operatorname{diag}(1,0,\cdot\cdot\cdot,0)$, $A$ is symmetric positive semi-definite and $S$ is an approximation of the first derivative at the boundary.
    \label{def:SBPproperty}
\end{definition}
We note, as the name suggests, that the SBP property is a discrete analogue of the \textit{integration-by-parts} technique used in the continuous setting. However, it does not fully mimic  the form of the continuous one. More specifically, to prove stability of the scheme, it will be useful to rewrite the matrix $A$ in Definition~\ref{def:SBPproperty} in terms of the first derivative operator $D_1$. Therefore, we shall need the notion of \textit{fully compatible} operators \cite{Mattsson4,Mattsson5}. 
\begin{definition}
    Let $D_1$ and $D_2$ be $2p$'th-order narrow-stencil SBP operators. If $D_2$ can be written in the following form
    $$D_2=H^{-1}(-D_1^THD_1-R+BS)$$
   and the remainder $R$ is symmetric positive semi-definite, then the operators $D_1$ and $D_2$ are said to be compatible. Moreover, they are said to be fully compatible if instead
    $$D_2=H^{-1}(-D_1^THD_1-R+BD_1).$$
    \label{def:fullycompatible}
\end{definition}
The remainder $R$ is shown to be positive semi-definite for orders $2p=2,4,6,8$ in Lemma 3.2 of \cite{Mattsson4}. Note that the second $D_2$ operator is obtained from the first one by directly replacing $S$ with $D_1$. However, as $S$ has a truncation error $\mathcal{O}(h^{p+1})$ and by definition $D_1$ instead has truncation error $\mathcal{O}(h^{p})$ on the boundary, using a fully compatible operator will lead to a reduction of one order on the boundary.
\\

We now turn to the SBP-SAT discretization of \eqref{eq:ContinuousFormulationU}-\eqref{eq:ContinuousFormulationV}. Let the domain $\Omega_u$ be rectangular and discretized uniformly by $m$ grid points in the $x$-direction and $n$ grid points in the $y$-direction. We introduce two convenient shorthands for the notation. First, operators in a non-calligraphic font with a lowered $x$ or $y$ denotes square operators with size corresponding to the indicated dimension. Using this notation, we let $I_x$ and $I_y$ be identity matrices of dimensions $m\times m$ and, $n\times n$ respectively. Analogously, let  $D_x,D_y$, $D_{xx},D_{yy}$ be first and second derivatives satisfying the SBP property and $S_x,S_y$ the derivative along the boundary in $x$ and $y$ directions. Also, let $H_x,H_y$ be the corresponding quadrature matrices and $B_x,B_y$ as in Definition~\ref{def:SBPproperty}.

Then, in order to adapt the operators to the two-dimensional problem, we utilize the tensor-product formulation. To distinguish these operators from the one-dimensional ones, we write them using calligraphic font, with a subscript indicating the dimension it acts along. An operator $\mathcal{F}$ acting along the $x$-direction would be written as $\mathcal{F}_x=I_2\otimes F_x\otimes I_y$, while it would be $\mathcal{F}_y=I_2\otimes I_x\otimes F_y$ in the $y$-direction. We refer to the calligraphic quantities as being of \textit{full-size}.

 Moreover, the following material property matrices will be required to formulate both the SBP-SAT discretization itself and the traction operators
 
\begin{align*}
    \LL_1&=\begin{bmatrix}
        (2\mu_u+\lambda_u)(I_x\otimes I_y) & 0 \\
        0 & \mu_u(I_x\otimes I_y)
    \end{bmatrix},\\
    \LL_2&=\begin{bmatrix}
        0 & \lambda_u(I_x\otimes I_y)\\
        \mu_u(I_x\otimes I_y) & 0
    \end{bmatrix},\\
    \LL_3&=\begin{bmatrix}
        \mu_u(I_x\otimes I_y) & 0 \\
        0 & (2\mu_u+\lambda_u)(I_x\otimes I_y)
    \end{bmatrix}.
\end{align*}

To formulate the traction, we introduce the following operators
$$\mathcal{T}_x=\LL_1\Dx+\LL_2\Dy\qquad\mathcal{T}_y=\LL_3\Dy+\LL_2^T\Dx.$$

so that the tractions on the boundaries are obtained by multiplying with $\pm\Bx,\pm\By$ to only pick up the values along the boundaries. This is done as the operators are defined on the entire mesh $\Omega_u$. Using the above notation, we then can formulate the SBP-SAT discretization for \eqref{eq:ContinuousFormulationU} as
\begin{equation}
    \rho_u\uu_{tt}=\LL_1\Dxx\uu+\LL_3\Dyy\uu+(\LL_2+\LL_2^T)\Dxy\uu+\alpha\Hx^{-1}\Bx\Tx\uu+\beta\Hy^{-1}\By\Ty\uu,
    \label{eq:FDformulation}
\end{equation}
where $\uu=[u_{11},...,u_{1n},...,u_{mn}]^T$ denotes the grid evaluation of the field with $u_{ij}\approx u(x_i,y_j,t)$ and $\alpha,\beta$ are penalty parameters determined by stability analysis.

Note that the last two terms correspond to the SAT-terms which impose the traction-free boundary condition. In the SBP-DG formulation derived in Section~\ref{sec: stability proof}, we only consider the contribution from the term that corresponds to the southern boundary and therefore the interface $\Gamma$. The SBP-SAT discretization \eqref{eq:FDformulation} is shown to be stable under the choice $\alpha=\beta=-1$ in \cite{Duru2}.

\subsection{Discontinuous Galerkin formulation}
\label{subsec: dG formulation}
Now consider the domain $\Omega_v$ as in the previously mentioned setup and let $\triangulation$ denote a  quasi-uniform triangulation  such that $\Omega_v=\bigcup_{K\in\triangulation}K$, where each element has diameter $h_K$. Moreover, let $h_v=\min_{K}h_K$. To discretize \eqref{eq:ContinuousFormulationV} we use the symmetric interior penalty discontinuous Galerkin method (IPDG) outlined in \cite{Dupont}. The approximation will be based on the following space
$$\testspace=\{\varphi : \varphi|_K\in\mathscr{P}_q(K),K\in\triangulation\},$$
where $\mathscr{P}_q(K)$ denotes the space of polynomials of order $q$ on the element $K$. Before formulating the IPDG form of the elastic wave equation, we must introduce a partition of the edges in the triangulation. Following the notation in \cite{Dupont}, let  $\edges=\{\gamma_i,i=1,...,n_e\}$ denote the set of all edges in $\triangulation$. Moreover, define an order such that $\gamma_i\subset\Omega_v$ for $i=1,..,n_i$ and $\gamma_i\subset \partial\Omega$ for $i=n_i+1,..,n_e$. Then the definition of \textit{interior edges} and \textit{boundary edges} follows naturally as the ordered subsets of $\edges$
\begin{align*}
    \interioredges&=\{\gamma_i : \gamma_i\in\edges, i=0,..,n_i\},\\
    \boundaryedges&=\{\gamma_i : \gamma_i\in \edges, i=n_i+1,...,n_e\}. 
\end{align*}
To couple the solution between neighbouring elements we shall need the notion of \textit{jump} and \textit{average} operators. Let $\gamma\in\interioredges$ be an edge shared by the elements $K^+,K^-\in\triangulation$ with unit outward pointing normals $\nvec^+,\nvec^-$ respectively and $\varphi$ an arbitrary piecewise smooth function.  Then the jump $\llbracket\varphi\rrbracket_\gamma$ and average $\{\varphi\}_\gamma$ across the edge $\gamma$ are defined depending on if it is an interior or boundary edge of $K^+$ accordingly
\begin{alignat*}{2}
    \llbracket\vvarphi\rrbracket_\gamma&=\varphi_+\nvec^+ +\varphi_-\nvec^-,\qquad \{\varphi\}_\gamma&=\frac{\varphi_++\varphi_-}{2}\qquad\text{ if }\gamma\in\interioredges,\\
    \llbracket\vvarphi\rrbracket_\gamma&=\varphi_ +\nvec^+,& \{\varphi\}_\gamma =\varphi_+,\qquad\text{ if }\gamma\in\boundaryedges.
\end{alignat*}
For piecewise smooth vector-valued functions $\pmb{\phi}$, the average is defined in analogous manner as $\{\pmb{\phi}\}_\gamma=\frac{\pmb{\phi}^++\pmb{\phi}^-}{2}$ for interior edges and $\{\pmb{\phi}\}_\gamma=\pmb{\phi}^+$ for boundary edges. \revone{ Likewise, we set the jump to be $\llbracket\pmb{\phi}\rrbracket_\gamma=\frac{\pmb{\phi}^+-\pmb{\phi}^-}{2}$ for interior edges and $\llbracket\pmb{\phi}\rrbracket_\gamma=\pmb{\phi}^+$ for boundary edges.}
\\

To pose the IPDG formulation, we reformulate \eqref{eq:ContinuousFormulationV} into vector form. Letting $\vv=[v_1,v_2]^T$ and using that $$\nabla\otimes\nabla=\begin{bmatrix}
    \partial_x^2 & \partial^2_{xy}\\
    \partial^2_{xy} & \partial_y^2,
\end{bmatrix}$$
we can reformulate \eqref{eq:ContinuousFormulationV} in terms of the field $\vv$ as
\begin{equation}
    \rho_v\partial_t^2\vv=\mu_v\Delta\vv+(\mu_v+\lambda_v)(\nabla\otimes\nabla)\vv.
\end{equation}
Multiplying with $\vvarphi\in\testspace\times\testspace$ and integrating over an element $K\in\triangulation$ allows us to formulate the problem locally on $K$ as
$$\int_K\rho_v\vvarphi\cdot\partial_t^2\vv_h dA=\int_K\mu_v\vvarphi\cdot\Delta\vv_h dA+\int_K(\mu_v+\lambda_v)\vvarphi\cdot[(\nabla\otimes\nabla)\vv_h] dA.$$
Summing the above equation over each element in $\triangulation$, integrating by parts and using the definition of jumps and averages, we can now state the IPDG formulation of the traction-free problem as; for each $t\in(0,T]$, find $\vv_h\in\testspace\times\testspace$ such that
\begin{align}
    \sum_{K_\in\triangulation}(\rho_v \partial_t^2\vv_h,\vvarphi)_K=&-\sum_{K\in\triangulation}\int_K\lambda_v(\nabla\cdot\vv_h)(\nabla\cdot\vvarphi)+\mu_v(\nabla\vv_h+\nabla\vv_h^T):\nabla\vvarphi dA\nonumber\\
    &
    +\sum_{\gamma\in\interioredges}\int_\gamma\{\ttau^\gamma(\vv_h)\}_\gamma\cdot\llbracket\vvarphi\rrbracket_\gamma+\{\ttau^\gamma(\vvarphi)\}_\gamma\cdot\llbracket\vv_h\rrbracket_\gamma\nonumber
    \\
    &-\frac{\alpha}{h}(2\mu_v+\lambda_v)\llbracket\vv_h\rrbracket_\gamma\cdot\llbracket\vvarphi\rrbracket_\gamma dl
    \label{eq:dgformulation}
\end{align}
for all $\vvarphi\in\testspace\times\testspace$. \revone{Here the operation $A:B=\sum_{ij}A_{ij}B_{ij}$ denotes matrix contraction}. The first row on the right-hand side is obtained by Green's identity, together with the first term on the second row. The following two terms are added to ensure symmetry and stability of the overall scheme. Notably, they vanish in limit when the mesh size goes to zero, thereby still approximating \eqref{eq:ContinuousFormulationV}. Finally, the stability parameter $\alpha$ needs to be chosen suitably large for \eqref{eq:dgformulation} to satisfy an energy estimate \cite{Antonietti}. A discussion on the relation between \eqref{eq:ContinuousFormulationV} and \eqref{eq:dgformulation} is presented in \cite{Basabe}.

 To end this section, we state two estimates needed for the stability analysis of the coupled scheme. First, we bound the terms on the right-hand side of the first row in \eqref{eq:dgformulation} by the norm of the stress tensor. For notational convenience let $\varepsilon(\vv)$ be the symmetric part of the Jacobian of $\vv$, called the \textit{strain-tensor}, then the following estimate holds.
 \begin{lemma}
\revone{Let the Lamé parameters $\mu_v,\lambda_v$ be positive and consider the case of plane strain}, then the $L^2$ inner-product of the stress tensor $\sigma(\vv)$ and strain tensor $\varepsilon(\vv)$ can be bounded in terms of the $L^2$-norm of $\sigma(\vv)$ as follows 
\begin{equation}
    \int_{\Omega_v}\sigma(\vv):\varepsilon(\vv)dA\geq\frac{1}{4\mu_v+2\lambda_v}\|\mathbb{\sigma}(\vv)\|_{L^2(\Omega_v)}^2,
\end{equation}

\label{lemma:contractionestimate}
 \end{lemma}

\begin{proof}
Using Hooke's law, we can write the strain tensor in terms of the stress tensor in \revone{Voigt} notation as $\eepsilon_v(\vv)=\mathbf{C}\ssigma_v(\vv)$ where $\mathbf{C}$ is the compliance matrix, $\ssigma_v(\vv)=[\sigma_{11}(\vv),\sigma_{22}(\vv),\sigma_{12}(\vv)]^T$ and $\eepsilon_v(\vv)=[\varepsilon_{11}(\vv),\varepsilon_{22}(\vv),2\varepsilon_{12}(\vv)]^T$. In the setting of plane strain, $\mathbf{C}$ is given by

\begin{equation}\mathbf{C}=\begin{bmatrix}
  \frac{2\mu+\lambda}{4\mu(\mu+\lambda)} & -\frac{\lambda}{4\mu(\mu+\lambda)} & 0\\
  -\frac{\lambda}{4\mu(\mu+\lambda)} & \frac{2\mu+\lambda}{4\mu(\mu+\lambda)} & 0\\
  0 & 0& \frac{1}{\mu},
\end{bmatrix}
\label{eq:compliancematrix}
\end{equation}
which is clearly symmetric and real from the definition of the Lamé parameters \cite{Larson}. The desired inequality can therefore be reformulated as finding a lower bound on a constant $k$ such that $\ssigma_v^T\mathbf{C}\ssigma_v\geq k\ssigma_v^T\ssigma_v$ holds. From the properties of $\mathbf{C}$ it is clear that $k$ is bounded by the smallest eigenvalue, which can be shown using \eqref{eq:compliancematrix} to yield $k\leq\frac{1}{2(\mu+\lambda)}$. Therefore, it must in particular hold for the choice $k=\frac{1}{4\mu+2\lambda}$ by the assumption on positivity of the Lamé parameters.
\end{proof}
\begin{lemma}
    Let $\Gamma\subset\partial\Omega$ and $h_v$ be the smallest element size in $\triangulation$, then there exist a constant $C>0$ only dependent on the quality of $\triangulation$ such that the stress tensor $\sigma(\vv_h)$ satisfies the following bound
    $$\|\sigma(\vv_h)\|_{\Omega_v}^2\geq Ch_v\|\sigma(\vv_h)\cdot\nvec\|_{L^2(\Gamma)}^2\quad\text{ for all }\vv_h\in\testspace\times\testspace.$$
\end{lemma}
\begin{proof}
    Let $\Gamma\subset\partial\Omega$ and introduce $\edges^\Gamma\subset\boundaryedges$ as the set of boundary edges having non-empty intersection with $\Gamma$. Moreover, let $\triangulation^\Gamma\subset\triangulation$ denote the set of triangles with at least one edge belonging to $\edges^\Gamma$ and $\|\vv_h\|^2_{L^2(K)}$ the $L^2$-norm of the discrete approximation on a triangle $K\in\triangulation$. Then there exist a positive constant $C'>0$ dependent on the quality of $\triangulation$ such that
    
    \begin{align*}\|\sigma(\vv_h)\cdot\nvec\|^2_{L^2(\Gamma)}&=\sum_{\gamma\in\edges^\Gamma}\|\sigma(\vv_h)\cdot\nvec\|^2_{L^2(\gamma)}\leq\sum_{K\in\triangulation^\Gamma}\|\sigma(\vv_h)\cdot\nvec\|^2_{L^2(\partial K)}\\
    &\leq \sum_{K\in\triangulation^\Gamma}\frac{C'}{h_K}\|\sigma(\vv_h)\|_{L^2(K)}^2
    \leq\frac{C'}{h_v}\sum_{K\in\triangulation}\|\sigma(\vv_h)\|_{L^2(K)}^2\\
    &=\frac{C'}{h_v}\|\sigma(\vv_h)\|^2_{L^2(\Omega_v)},
    \end{align*}
    for $\vv_h\in\testspace\times\testspace$. To obtain the first inequality on the second row we use an inverse-type estimate on a reference element followed by a scaling argument. The second inequality follows from $h_v=\min_{K\in\triangulation}h_K$. Hence we can conclude that there exist a positive constant $C>0$ such that
    $$\|\sigma(\vv)\|^2_{L^2(\Omega_v)}\geq Ch_v\|\sigma(\vv)\cdot\nvec\|^2_{L^2(\Gamma)}$$
    and the lemma follows.
    \label{lemma:inverseestimate}
\end{proof}

\section{The FD-DG discretization}
\label{sec: the coupled scheme}
In this section, we shall present an in-depth discussion of the methodology used to couple the two grids, as well as the construction and properties of the projection operators used. Then we end the section by proving stability of the semi-discretization by showing that it satisfies a discrete energy estimate.
\subsection{The FD-DG interface and projection operators}
\label{sec: interface}
To couple the FD and DG domains introduced in Section~\ref{sec: the elastic wave equation} we use the notion of norm compatible operators developed by Kozdon and Wilcox in \cite{Kozdon}. The assumption of norm compatibility will prove to be crucial when deriving an energy estimate. In \cite{Kozdon}, the model problem was the wave equation in first order form. When applied to a second-order formulation with non-conforming FD-FD interfaces, however, using the operators from \cite{Kozdon} resulted in suboptimal convergence \cite{Lundquist}. To recover optimal convergence, we therefore use the notion of projection pairs introduced in \cite{Almquist2}.

As in the previous section, let $\Gamma=\partial\Omega_u\cap\partial\Omega_v$ denote the interface. Moreover, let $\fdgrid$ be the uniform FD nodes along $\Gamma$, ordered left to right, and $\mathscr{T}_{\tilde{u}}$ the mesh whose element boundaries are defined by the nodes in $\fdgrid$. Similarly, we let $\dggrid$ denote the edges in the DG mesh along the interface, and $\gluegrid$ the mesh with both FD and DG nodes as element boundaries. The setup of these grids are  illustrated in Figure~\ref{fig:meshes}. For completeness, we also introduce the corresponding function spaces $\mathscr{F}_{\tilde{u}}$, $\mathscr{F}_g$ and $\mathscr{F}_v$ of piecewise continuous polynomials of order $q$ over $\mathscr{T}_{\tilde{u}}$,$\gluegrid$ and $\dggrid$ respectively.

\begin{figure}[h!]
    \centering
  \includegraphics[width=0.7\textwidth]{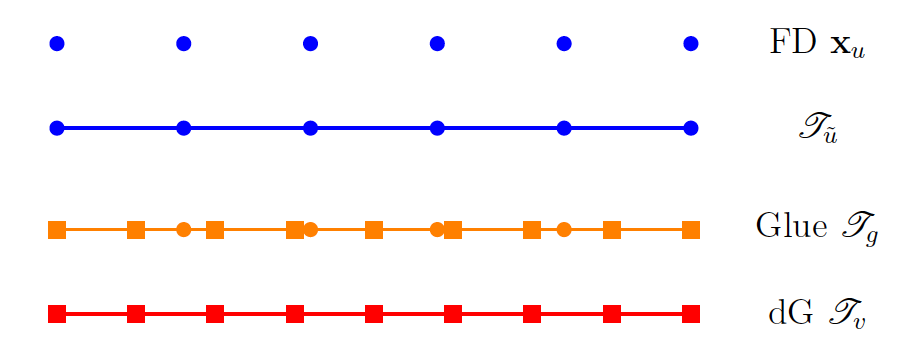}
    \caption{An illustration of the FD mesh $\fdgrid$, the intermediary piecewise continuous grid $\mathscr{T}_{\Tilde{u}}$, glue grid $\mathscr{T}_g$ and DG grid $\mathscr{T}_v$. Although the piecewise continuous intervals are drawn as if the nodes are overlapping, each element should be thought of being disjoint with repeated edge nodes.}
    \label{fig:meshes}
\end{figure}

Now let $f$ be a smooth function defined on $\Gamma$, with $\ff_u$ representing its grid evaluation on $\fdgrid$. We introduce the projection operator $\Pi_{u2v}$ such that $\Pi_{u2v}\ff_u$ contains the coefficients of a piecewise continuous polynomial $\mathscr{P}f$ on $\dggrid$ that approximates $f$. Let $\{\varphi^\gamma_j\}_{j=0}^q$ denote a set of local nodal basis functions in $\mathscr{F}_v$ associated with element $\gamma$ in $\dggrid$, i.e. for a node $x_k$ in the element $\gamma$, $\varphi_j^\gamma(x_k)=\delta_{jk}$, where $\delta_{jk}$ denotes the Dirac-delta. Moreover, let  $(\Pi_{u2v}\ff_u)^\gamma_j$ denote the projected value at node $j$ in element $\gamma$. Then, since the basis is nodal,  $\mathscr{P}f(x)$ can be expanded in the element $\gamma$ as $\mathscr{P}f(x)|_\gamma=\sum_{j=0}^q(\Pi_{u2v}\ff)^\gamma_j\varphi_j^\gamma(x)$. Similarly, we can construct an operator $\Pi_{v2u}$, transforming the coefficients of the piecewise polynomial back to the FD grid. As demonstrated in Figure~\ref{fig:meshes}, we do not force the nodes in $\fdgrid$ and $\dggrid$ to align. To remedy this, we introduce the aforementioned glue grid, $\gluegrid$ having the nodes in $\fdgrid$ and $\dggrid$ as element boundaries. To this end, we can first transform pointwise values from $\fdgrid$ to piecewise continuous polynomials on $\mathscr{T}_{\Tilde{u}}$ via the operator $\Pi_{u2\tilde{u}}$. \revthree{The operator $\Pi_{u2\tilde{u}}$ is therefore in general not exact. However, note that both $\mathscr{F}_v$ and $\mathscr{F}_{\tilde{u}}$ are subspaces of $\mathscr{F}_g$}. We can therefore construct an exact projection operator $\Pi_{g2\tilde
u}$ from $\mathscr{F}_g$ to $\mathscr{F}_{\tilde{u}}$, followed by another similar operator $\Pi_{g2v}$ from $\mathscr{F}_g$ to $\mathscr{F}_v$. The product of these then forms our compound projection operators $\Pi_{u2v}=\Pi_{g2v}\Pi_{\tilde{u}2g}\Pi_{u2\tilde{u}}$ and $\Pi_{v2u}=\Pi_{\tilde{u}2u}\Pi_{g2\tilde{u}}\Pi_{v2g}$. In order to prove stability of the coupled scheme, we utilize the following notion of \textit{norm-compatible} operators.

\begin{definition}
    Let $\mathscr{T}_k$ and $\mathscr{T}_l$ be two meshes with associated function spaces $\mathscr{F}_k,\mathscr{F}_l$. Moreover, let $M_k$ and $M_l$ denote the mass matrices of the function spaces. Then the pair of projection operators $\Pi_{k2l},\Pi_{l2k}$ are said to be norm compatible if it holds
    $$M_l\Pi_{k2l}=(M_k\Pi_{l2k})^T.$$
    \label{def:normcompatibility}
\end{definition}
We construct $\Pi_{u2\tilde{u}}$ such that the pair $\Pi_{u2\tilde{u}},\Pi_{\tilde{u}2u}$ is norm compatible. Then, as the remaining projection operators correspond to projections between subspaces, they are norm compatible by construction. Since $\mathscr{F}_{\tilde{u}}$ and $\mathscr{F}_v$ are subspaces of the glue-space $\mathscr{F}_g$, the following result is implied.
\begin{corollary}
    Let $\Pi_{u2v}=\Pi_{g2v}\Pi_{\tilde{u}2g}\Pi_{u2\tilde{u}}$ and $\Pi_{v2u}=\Pi_{\tilde{u}2u}\Pi_{g2\tilde{u}}\Pi_{v2g}$, then the pair $(\Pi_{u2v},\Pi_{v2u})$ necessarily satisfies
    $$M_v\Pi_{u2v}=(H_u\Pi_{v2u})^T.$$
\end{corollary}

Analogous to the norm compatibility, we have to pose accuracy conditions on the operators. We begin with the accuracy of the projection operators $\Pi_{u2\tilde{u}}$ and $ \Pi_{\tilde{u}2u}$. Let $\tilde{\ff}$ be the coefficients obtained when expanding $f$ along $\mathscr{T}_{\tilde{u}}$ using the basis $\{\varphi_i^{\gamma}(x)\}$ for each $\gamma\subset\mathscr{T}_ {\tilde{u}}$ and introduce 

$$\ee_{\Tilde{u}2u}=\Pi_{\tilde{u}2u}\Tilde{\ff}-\ff\qquad \ee_{u2\tilde{u}}=\Pi_{u2\tilde{u}}\ff-\tilde{\ff}$$
as projection errors arising when projecting between $\fdgrid$ and $\mathscr{T}_{\tilde{u}}$. If we use a polynomial basis, it is natural to require that the errors vanish for polynomials up to a given order $q$. The projection operator will use a centered stencil in the interior, but a one-sided stencil near the boundary. Therefore, we must also distinguish between the boundary order of accuracy $q_\partial(\Pi_{u2v})$ and the interior order $q_i(\Pi_{u2v})$. \revone{ Notably, using a method of interior order $q_i(\Pi_{u2v})$ makes the projection errors vanish for all polynomials up to order $q_i(\Pi_{u2v})-1$. Moreover, it was shown in \cite{Lundquist} that for norm compatibility to hold between two FD grids, the overall orders of accuracy on the boundary of the projection operators when using a quadrature rule of $2p$ should be related by $q_\partial(\Pi_{u2v})+q_\partial(\Pi_{v2u})\leq 2p+1$. For a deeper discussion, see \cite{Almquist2,Wang1}}.

The remaining projection operators $\Pi_{\tilde{u}2g},\Pi_{g2v},\Pi_{v2g}$ and $\Pi_{g2\tilde{u}}$ all correspond to transformations between function spaces which can be seen as subspaces of $\mathscr{F}_g$. Since the basis functions all are of the same degree, the projection orders do not affect the order of accuracy. Thereby, the order of the overall projection operators $\Pi_{u2v}$ and $\Pi_{v2u}$ is completely determined by the accuracy of $\Pi_{u2\tilde{u}}$ and $\Pi_{\tilde{u}2u}$. 

We shall follow the projection method outlined in \cite{Wang1} of projection pairs of \textit{good} and \textit{bad} operators $(\Pi^g_{u2\tilde{u}},\Pi^b_{\tilde{u}2u})$, $(\Pi^b_{u2\tilde{u}},\Pi^g_{\tilde{u}2u})$, where superscript $g$ denotes the good operator and $b$ the bad one, which also are norm compatible. The notion of good and bad relates to the orders of accuracy. From the discussion above on the bound of the order of accuracy we let the orders for the good and bad orders be given by \revone{$q_\partial(\Pi^g)=p+1, q_i(\Pi^g)=2p$  and $q_\partial(\Pi^b)=p, q_i(\Pi^b)=2p$} when using an FD method of order $2p$ and DG method with basis of order $p+1$. Since the pairs are norm compatible, it is sufficient to set up a system of equations to solve for one operator from each pair and then use norm-compatibility to construct the remaining one. This procedure yields a set of free parameters, which in \cite{Wang1} were chosen by the authors to optimize accuracy. In this work, we instead optimize over these parameters with respect to both accuracy and spectrum, demonstrating a significant improvement as compared to choosing some arbitrary parameters.

\subsection{Stability of the coupled method}
\label{sec: stability proof}
We now set to prove the stability of the FD-DG scheme. Recall that we consider the domain $\Omega=\Omega_u\cup\Omega_v$ with interface $\Gamma=\partial\Omega_u\cap\partial\Omega_v$. Moreover, we let each of the domains have constant densities $\rho_u,\rho_v$ and Lamé parameters given by $\mu_u,\lambda_u,\mu_v,\lambda_v$ and assume that they are non-negative. 

To construct operators which are of full-size, let $N_u$ and $N_v$ denote the numbers of DOFS in the FD and DG meshes respectively. Then introduce the operators $P_{u2v}^g,P_{u2v}^b$ and $P_{v2u}^b,P_{v2u}^b$ as matrices of sizes $N_v\times N_u$ and $N_u\times N_v$ with entries given the corresponding one-dimensional projection operators $\Pi^g_{u2v},\Pi^b_{u2v},\Pi^g_{v2u}$ and $\Pi^b_{v2u}$ from Section~\ref{sec: interface}, with the entries ordered consistent with the order of nodes along the interface.  Recalling the notation convention from Section \ref{subsec: SBP formulation}, we define the full-size projection operators as
\begin{equation*}
\Puvg=I_2\otimes P_{u2v}^g\quad\Puvb=I_2\otimes P_{u2v}^b\quad
\Pvug=I_2\otimes P_{v2u}^g\quad\Pvub=I_2\otimes P_{v2u}^b.\end{equation*}

Moreover, we utilize the shorthand that any quantity with a subscript $\Gamma$ denotes the degrees of freedom along the interface, so e.g., $\Bg=I_2\otimes I_x\otimes B_0$, $\uu_\Gamma=\Bg\uu$ and $\vv_\Gamma$ is the evaluation of $v$ on the DG mesh, ordered analogous to the projection operators, with any value outside $\Gamma$ set to zero. We then have the following semi-discretization for the FD side of the FD-DG scheme

\begin{align}
\begin{split}
\rho\uu_{tt}&=\LL_1\Dxx\uu+\LL_3\Dyy\uu+(\LL_2+\LL_2^T)\Dxy\uu\\
&+\HH^{-1}\bigg[-\frac{1}{2}\Bg\Hx(\Tu\uu+\Puvb\Tv\vv)+\frac{1}{2}\Tu^T\Hx(\uu_\Gamma-\Puvg\vv_\Gamma)\bigg]\\
&-\HH^{-1}\bigg[\chi_1\frac{\alpha_v}{h_v}\Bg\Hx\Puvb(\Pvug\uu_\Gamma-\vv_\Gamma)+\chi_2\frac{\beta_u}{h_u}\Bg\Hx(\uu_\Gamma-\Puvg\vv_\Gamma)\bigg].\\
\label{eq:FDdiscretization}
\end{split}
\end{align}

Here the first row corresponds to the bulk and mimics that of \eqref{eq:FDdiscretization}. The remaining terms are SATs enforcing the continuity conditions \eqref{eqn: contcond}-\eqref{eqn: traccond}. We introduce $\Tu=-\mathcal{B}_\Gamma\mathcal{T}_y$ and $\Tv$ as the full-size operator containing the traction ordered along $\Gamma$ on the FD and DG meshes respectively. Moreover, $\alpha_u$ and $\beta_u$ represent dimensionless stability parameters that must be determined, while $\chi_1$ and $\chi_2$ are constants with dimensions of pressure. To maintain dimensional consistency, the two last terms must also be scaled by the mesh sizes $h_u$ and $h_v$, which, as mentioned earlier, may not match along the interface. Note that here $h_u$ denotes the size of the FD mesh in the $x$-direction and $h_v$ the smallest element size in $\triangulation$.

We now turn to the DG semi-discretization and take inspiration of the form \eqref{eq:dgformulation}. However, we only consider the contribution to the fluxes of the edges in the set $\edges^\Gamma$, consisting of edges having non-zero intersection with $\Gamma$. This is possible since the internal fluxes combine with the bulk terms analogously as in the DG case, and hence can be neglected. Moreover, for notational convenience, we rewrite the bulk in terms of the strain tensor $\varepsilon(\vv)=\frac{1}{2}(\nabla\vv+\nabla\vv^T)$ and let $(\cdot,\cdot)_{\Omega_v}$ denote the standard $L^2(\Omega_v)$ inner product. We then get the coupled semi-discretization as: find $\vv\in\testspace\times\testspace$ such that for each time $t\in(0,T]$ it holds

\begin{align}
(\rho_v\vv_{tt},\pphi)_{\Omega_v}=&-\lambda_v(\nabla\cdot\vv,\nabla\cdot \pphi)_{\Omega_v}-2\mu_v(\varepsilon(\vv),\varepsilon(\pphi))_{\Omega_v}\nonumber\\
&+\frac{1}{2}\sum_{\gamma\in\edges^\Gamma}(\vv_\Gamma-\uu^\gamma_\Gamma,\ttau^\gamma(\pphi))_\gamma
+\frac{1}{2}\sum_{\gamma\in\edges^\Gamma}(\ttau^\gamma(\vv)-\ttau^\gamma(\uu),\pphi)_\gamma\nonumber\\
&-\chi_1\frac{\alpha_v}{h_v}\sum_{\gamma\in\edges^\Gamma}\big(\vv_\Gamma-\uu_\Gamma^\gamma,\pphi\big)_\gamma-\chi_2\frac{\beta_v}{h_u}\sum_{\gamma\in\edges^\Gamma}\big(\vv_\Gamma^\flat-{\Tilde{\uu}}^\gamma_\Gamma,\pphi\big)_\gamma
\label{eq:dGdiscretization}
\end{align}
for every $\pphi\in\testspace\times\testspace$. Here we have introduced the notation $\vv_\Gamma^\flat=\sum_{i=0}^q(\Puvb\Pvug\vv_{\Gamma})^\gamma_i\pphi_i^\gamma$, and $\uu_\Gamma^\gamma=\sum_{i=0}^q(\Puvg\uu_\Gamma)^\gamma_i\pphi_i^\gamma$ for the component of $\uu_\Gamma$ projected from $\fdgrid$ onto $\gamma\in\edges^\Gamma$, where $\{\pphi^\gamma_i(x)\}_{i=0}^q$ again denotes the local nodal basis on the edge $\gamma$. Analogously we introduce the projected quantities $\ttau^\gamma(\uu)=\sum_{i=0}^q(\Puvb\Tu\uu)_i^\gamma\pphi^\gamma_i$ and $\Tilde{\uu}^\gamma_\Gamma=\sum_{i=0}^q(\Puvb\uu_\Gamma)^\gamma_i\pphi_i^\gamma$, where $\ttau^\gamma(\uu)$ corresponds to the FD traction projected onto the edge $\gamma$ and $\Tilde{\uu}_\Gamma$ is a projection onto $\gamma$ by the bad projection operator. We recall that the nodes are ordered in a way consistent with the order of $\fdgrid$. With this in place. we can set to prove the main theoretical contribution of the paper.
\begin{theorem}
    Let $\chi_1=2\mu_v+\lambda_v$ and $\chi_2=2\mu_u+\lambda_u$, where the Lamé parameters are positive. Assume that the operators $\Dy$, $\Dyy$ are fully compatible and the projection operators between the FD and DG mesh are norm compatible, then the semi-discretizations \eqref{eq:FDformulation} and \eqref{eq:dGdiscretization} form an energy conserving scheme  under the choices $\alpha_u=\alpha_v\geq\frac{1}{4C}$ and $\beta_u=\beta_v\geq\frac{1}{4C'}$ for some constants $C,C'>0$, with $C$ depending on the quality of $\triangulation$ and $C'$ on the order of the SBP operators used.
\end{theorem}
\begin{proof}
    We begin by considering the discretization \eqref{eq:dGdiscretization}. In order to write the semi-dicretization on a form compatible with \eqref{eq:FDdiscretization} we need to introduce the full-size one-dimensional mass matrix $M_\Gamma$ along the interface and $\MM=I_2\otimes M_\Gamma$ such that $(\pphi_i,\pphi_j)_{\Gamma}=\pphi_i^T\MM\pphi_j$ for $\pphi_i,\pphi_j\in\testspace\times\testspace$. We also need an operator $\Tv$ for the DG discretization such that for a given edge $\gamma\in\edges$ we can represent the traction as $\ttau^\gamma(\vv)=\sum_{i=0}^q(\Tv\vv)^\gamma_i\vvarphi_i$. Then, letting $\pphi=\vv_t$ we can write the semi-discretization as
    \begin{align*}
       \frac{1}{2}\frac{d}{dt}(\rho_v\vv_t,\vv_t)_{\Omega_v}=&-\frac{1}{2}\frac{d}{dt}(2\mu\|\varepsilon(\vv)\|^2_{\Omega_v}+\lambda_v\|\nabla\cdot\vv\|^2_{\Omega_v})\nonumber\\
       &+\frac{1}{2}(\vv_\Gamma-\Puvg\uu_\Gamma)^T\MM\Tv\vv_t+\frac{1}{2}(\Tv\vv-\Puvb\Tu\uu)^T\MM\vv_{\Gamma,t}\nonumber\\
       &-\chi_1\frac{\alpha_v}{h_v}(\vv_\Gamma-\Puvg\uu_\Gamma)^T\MM\vv_{\Gamma,t}\\
       &-\chi_2\frac{\beta_v}{h_u}(\Puvb\Pvug\vv_\Gamma-\Puvb\uu_\Gamma)^T\MM\vv_{\Gamma,t}.
    \end{align*}
    In order to proceed, recall the notion of norm-compatibility, relating the projection operators by 
   \begin{equation}
    \MM\Puvg=(\Hx\Pvub)^T\quad\text{and}\quad \MM\Puvb=(\Hx\Pvug)^T.
    \label{eq:normcompatibility}
    \end{equation}
    This allows us to rewrite the bad projection operators in terms of the good ones. \revone{ Gathering terms only containing $\vv$ by themselves and terms with both $\uu,\vv$ by themselves yields}
    \begin{align}
         \frac{1}{2}\frac{d}{dt}(\rho_v\vv_t,\vv_t)_{\Omega_v}&=-\frac{1}{2}\frac{d}{dt}\big[-\vv_\Gamma^T\MM\Tv\vv+\chi_1\frac{\alpha_v}{h_v}\vv_\Gamma^T\MM\vv_\Gamma\nonumber\\
         &+\chi_2\frac{\beta_v}{h_u}(\Pvug\vv_\Gamma)^T\Hx\Pvug\vv_\Gamma\big]\nonumber
         -\frac{1}{2}(\Puvg\uu_\Gamma)^T\MM\Tv\vv_t\nonumber\\
         &-\frac{1}{2}(\Tu\uu)^T\Hx\Pvug\vv_{\Gamma,t}+\chi_1\frac{\alpha_v}{h_v}(\Puvg\uu_\Gamma)^T\HH\vv_{\Gamma,t}\nonumber\\
         &+\chi_2\frac{\beta_v}{h_u}\uu_\Gamma^T\MM\Pvug\vv_{\Gamma,t}.
         \label{eq:secondstepdG}
    \end{align}
    We now turn to the FD discretization given by \eqref{eq:FDdiscretization}. Mimicking the standard integration by parts procedure, multiplying with $\uu_t^T\HH$ we get
    \begin{align}
        \rho_u\uu_t^T\HH\uu_{tt}&=\uu_t^T\HH\LL_1\Dxx\uu+\uu_t^T\LL_3\HH\Dyy\uu+\uu_t^T\HH(\LL_2+\LL_2^T)\HH\Dxy\uu\nonumber\\
        &-\frac{1}{2}\uu_{\Gamma,t}^T\Bg\Hx(\Tu\uu+\Puvb\Tv\vv)+\frac{1}{2}(\Tu\uu_t)^T\Hx(\uu_\Gamma-\Puvg\vv_\Gamma)\nonumber\\
        &-\chi_1\frac{\alpha_u}{h_v}\uu_{\Gamma,t}^T\Hx\Puvb(\Pvug\uu_\Gamma-\vv_\Gamma)-\chi_2\frac{\beta_u}{h_u}\uu_{\Gamma,t}^T\Hx(\uu_\Gamma-\Puvg\vv_\Gamma).
        \label{eq:firststepFD}
    \end{align}
    Before proceeding, we can note that it is possible to rewrite the first row above using the SBP property together with the assumption that the operators $\Dy,\Dyy$ are fully compatible. This yields
    \begin{align}
        \uu_t^T\HH\LL_1\Dxx\uu&+\uu_t^T\LL_3\HH\Dyy\uu+\uu_t^T\HH(\LL_2+\LL_2^T)\HH\Dxy\uu=-\uu_t^T\LL_1\Hy\Ax\nonumber\\
        &-\uu_t^T\LL_3\Hx\Ay-\uu_t^T\Hx\Bg\Dy\uu+\uu_t^T(\LL_2+\LL_2^T)\HH\Dxy\uu\nonumber\\
        &=-\uu_t^T\LL_1\Hy\Ax\uu-\uu_t^T\LL_3\Hx\Ay\uu-\uu_t^T\Hx\Bg\Dy\uu\nonumber\\
        &-\uu_t^T\LL_2^T(\Bg+(\Hy\Dy)^T)\Hx\Dx-\uu_t^T\Dx^T\HH\Dy\uu=(\star)
        \label{eq:rewritingFD}
    \end{align}
    where we also have neglected the terms not present on the boundary, e.g. the $\mathcal{B}_x$ term. In order to obtain a symmetric term which will be useful in the energy analysis we then gather terms and use compatibility to rewrite one step further, yielding
    \begin{align}
        (\star)&= -\uu_t^T\Hx\Tu\uu-\uu_t^T\big(\LL_2\Dx^T\HH\Dy+\LL_2^T\Dy^T\HH\Dx+\LL_1\Dx^T\HH\Dx\\\nonumber
        &+\LL_3\Dy^T\HH\Dy\big)\uu=-\uu_t^T\Hx\Tu\uu-\uu_t^T\UU\uu,
    \end{align}
    where we for notational convenience introduce the matrix 
    
    $$\UU=\LL_2\Dx^T\HH\Dy+\LL_2^T\Dy^T\HH\Dx+\LL_1\Dx^T\HH\Dx+\LL_3\Dy^T\HH\Dy$$, which by the virtue of $\LL_1,\LL_2$ and $\LL_3$ being block diagonal satisfies $\UU=\UU^T$ . Combining this with the last two lines in \eqref{eq:firststepFD} we can once more gather $\uu$ and $\vv$ terms as well as writing everything in terms of good projection operators by using \eqref{eq:normcompatibility}, resulting in
    \begin{align}
        \frac{1}{2}\rho_u\uu_t^T\HH\uu_t=&-\frac{1}{2}\big[\uu^T\Hx\Tu\uu+\uu^T\UU\uu+\chi_1\frac{\alpha_u}{h_v}(\Puvg\uu_\Gamma)^T\MM\Puvg\uu_\Gamma\nonumber\\
        &+\chi_2\frac{\beta_u}{h_u}\uu_\Gamma^T\Hx\uu_\Gamma\big]-\frac{1}{2}(\Puvg\uu_{\Gamma,t})^T\MM\Tv\vv-\frac{1}{2}(\Tu\uu_t)^T\Hx\Pvug\vv_\Gamma\nonumber\\
        &+\chi_1\frac{\alpha_u}{h_v}(\Puvg\uu_{\Gamma,t})^T\MM\vv_\Gamma+\chi_2\frac{\beta_u}{h_u}\uu_{\Gamma,t}^T\Hx\Pvug\vv_\Gamma
        \label{eq:secondstepFD}.
    \end{align}
To prove stability it now remains to show that the combination of \eqref{eq:secondstepFD} and \eqref{eq:secondstepdG} results in a non-positive quantity. Upon comparing the two equations it is immediate that by setting $\alpha_u=\alpha_v$ and $\beta_v=\beta_u$ we can add the two equations to get

\begin{equation}
    \frac{1}{2}\frac{dE}{dt}=-\frac{1}{2}\frac{d}{dt}\big(U_1+U_2\big),
\end{equation}
where we have introduced \revone{ the quantities $E=\rho_u\uu_t^T\HH\uu_t+(\rho_v\vv_t,\vv_t)_{\Omega_v}$},
\begin{align*}
    U_1=2\mu_v\|\varepsilon(\vv)\|^2_{\Omega_v}&+\lambda_v\|\nabla\cdot\vv\|^2_{\Omega_v}+\chi_1\frac{\alpha_v}{h_v}(\Puvg\uu_\Gamma)^T\MM\vv_\Gamma\\
    &+\chi_1\frac{\alpha_v}{h_v}(\Puvg\uu_\Gamma)^T\MM\Puvg\uu_\Gamma
    -2\chi_1\frac{\alpha_v}{h_v}(\Puvg\uu_\Gamma)^T\MM\vv_\Gamma\\
    &-\vv_\Gamma\MM\Tv\vv+(\Puvg\uu_\Gamma)^T\MM\Tv\vv    
\end{align*}
and
\begin{align*}
    U_2=\uu^T\UU\uu+\chi_2\frac{\beta_u}{h_u}\uu_\Gamma^T\HH_x\uu_\Gamma&+\chi_2\frac{\beta_u}{h_u}\uu_\Gamma^T\HH_x\Pvug\vv_\Gamma-\chi_2\frac{2\beta_u}{h_u}\uu_\Gamma^T\HH_x\Pvug\vv_\Gamma\\
&-\uu_\Gamma^T\Hx\Tu\uu+(\Tu\uu)^T\Hx\Pvug\vv_\Gamma.
\end{align*}
\revone{If we then can show that $U_1,U_2\geq 0$ it follows that $\tilde{E}=E+U_1+U_2$ corresponds to a semi-discrete energy which is conserved}. We set to show that $U_1\geq 0$. By completing the square we can rewrite it as 
\begin{align}
    U_ 1&=2\mu_v\|\varepsilon(\vv)\|^2_{\Omega_v}+\lambda_v\|\nabla\cdot\vv\|^2_{\Omega_v}-\frac{h_v}{4\alpha_v\chi_1}(\Tv\vv)^T\MM\Tv\vv\nonumber\\
    &+\chi_1\frac{\alpha_v}{h_v}\big(\vv_\Gamma-\Puvg\uu_\Gamma+\frac{h_v}{2\alpha_v\chi_1}\Tv\vv\big)^T\MM\big(\vv_\Gamma-\Puvg\uu_\Gamma+\frac{h_v}{2\alpha_v\chi_1}\Tv\vv\big).  
    \label{eq:E1}
\end{align}
Note that the second row is non-negative under the choice of $\chi_1$ and $\alpha_v$ positive, hence it is sufficient to consider the first row. This will follow from Lemma~\ref{lemma:contractionestimate} and Lemma~\ref{lemma:inverseestimate} since $2\mu_v\|\varepsilon(\vv)\|^2_{\Omega_v}+\lambda_v\|\nabla\cdot\vv\|^2_{\Omega_v}=\int_{\Omega_v}\sigma(\vv):\varepsilon(\vv)dA$, yielding
\begin{align}
 U_1&\geq h_v\bigg(\frac{C}{2\mu_v+\lambda_v}-\frac{1}{4\alpha_v\chi_1}\bigg)(\Tv\vv)^T\MM\Tv\vv\\
 &+\chi_1\frac{\alpha_v}{h_v}\big(\vv_\Gamma-\Puvg\uu_\Gamma+\frac{h_v}{2\alpha_v\chi_1}\Tv\vv\big)^T\MM\big(\vv_\Gamma-\Puvg\uu_\Gamma+\frac{h_v}{2\alpha_v\chi_1}\Tv\vv\big),    
 \label{eq:E1FinalForm}
\end{align}
where we have used that $\|\sigma(\uu)\cdot\nvec\|_{\Gamma}^2=(\Tv\vv)^T\MM\Tv\vv$. Now by setting $\chi_1=2\mu_v+\lambda_v$ we get that the expression for $U_1$ is non-negative if we assume that $\alpha_v\geq\frac{1}{4C}$, as desired. Finally we turn to $U_2$. Here it is also possible to rewrite the expression by means of completing a square to obtain
\begin{align}
    U_2&=\uu^T\UU\uu-\frac{h_u}{4\beta_u\chi_2}(\Tu\uu)^T\Hx\Tu\uu\nonumber\\
    &+\chi_2\frac{\beta_u}{h_u}\big(\uu_\Gamma-\Pvug\vv_\Gamma+\frac{h_u}{2\beta_u}\Tu\uu\big)^T\HH_x\big(\uu_\Gamma-\Pvug\vv_\Gamma+\frac{h_u}{2\beta_u}\Tu\uu\big),
    \label{eq:U2}
\end{align}
where it again is sufficient to bound the terms in the first row of $U_2$. \revone{A key step to obtain a bound is the positivity property from section 4.2 in \cite{Almquist1} for the SBP norm, which is an analogue to the trace inequality Lemma~\ref{lemma:inverseestimate} used on the DG side. In our case, the estimate is given by}
\begin{equation}
\uu^T\HH\uu\geq C'h_u\uu^T\Hx\uu,
\label{eq:SBPestimate}
\end{equation}
where $C'>0$ depends on the value of the quadrature matrix $H$ at the boundary and is given for SBP operators of order $2p=2,4,6$ in Table~\ref{table:coefficients}. 

\begin{table}
\centering
\begin{tabular}{|c|c|}
\hline
    Order $(2p)$ & $C'$  \\
    \hline
     2 &  0.5 \\
     \hline
     4 & 0.3542\\
     \hline
     6 & 0.3159\\
     \hline
    
     \end{tabular}  
     \caption{Value of the constant $C'$ used in the estimate \eqref{eq:SBPestimate} for SBP operators of order 2,4 and 6.}
       \label{table:coefficients}
\end{table}

\revone{To prove that $U_2\geq 0$ we follow the stability proof for the fully compatible case performed in \cite{Duru1}; rewriting the top row of $U_2$ as a large system on the form $[\Dx\uu,\Dy\uu]^T\mathcal{Q}[\Dx\uu,\Dy\uu]$ and show that $\mathcal{Q}$ is positive semi-definite. After some algebra, we obtain the first row of $\eqref{eq:U2}$ as}

\begin{equation}
    h_u\begin{bmatrix}
            \Dx\uu \\
            \Dy\uu
        \end{bmatrix}^T\begin{bmatrix}
(C'\LL_1-\frac{1}{4\beta_u\chi_2}\LL_2\LL_2^T)\HH_x & (C'\LL_2-\frac{1}{4\beta_u\chi_2}\LL_2\LL_3)\HH_x \\
(C'\LL_2^T-\frac{1}{4\beta_u\chi_2}\LL_2^T\LL_3^T)\HH_x & (C'\LL_3-\frac{1}{4\beta_u\chi_2}\LL_3\LL_3^T)\HH_x
\end{bmatrix}\begin{bmatrix}
            \Dx\uu \\
            \Dy\uu
        \end{bmatrix}
        \label{eq:u2energy}.
\end{equation}
In order to prove that this expression is non-negative, we use the method of Schur complements. From \cite{Zhang1} we recall that the compound matrix $\ZZ=[\ZZ_{11}, \ZZ_{12};\ZZ_{12}^T,\ZZ_{22}]$ is positive semi-definite if it holds that
\\
\begin{enumerate}
    \item the matrix $\ZZ_{22}$ is invertible,
    \item the matrix $\ZZ_{22}$ is positive semi-definite,
    \item the complement $\ZZ/\ZZ_{22}=\ZZ_{11}-\ZZ_{12}^T\ZZ_{22}^{-1}\ZZ_{12}$ is positive semi-definite.\\
\end{enumerate}

It is proven in \ref{section:Appendix} that the conditions above are met for the choice $\chi_2=2\mu_u+\lambda_u$ if the penalty parameter $\beta_v$ satisfies the bound $\beta_u>\frac{1}{4C'}$, as desired. Hence we have proven that \eqref{eq:u2energy} is non-negative and therefore we can conclude that $U_2\geq 0$ under the assumption $\beta_u\geq\frac{1}{4C'}$ which together with \eqref{eq:E1FinalForm} being non-negative if $\alpha_v\geq\frac{1}{4C}$ proves that
\revone{$$\frac{1}{2}\frac{d}{dt}(\rho_u\uu_t^T\HH\uu_t+(\rho_v\vv_t,\vv_t)_{\Omega_v}+U_1+U_2)= 0,$$}
under the desired bounds for $\alpha_u$, $\beta_v$ and thereby also the theorem.
\end{proof}
\begin{remark}
    It should be noted that the proof will fail at \eqref{eq:rewritingFD} if we remove the assumption of full compatibility as we cannot obtain a total time derivative. This assumption can be relaxed for the wave equation \cite{Almquist2020}. 
\end{remark}


\section{Numerical results}
\label{sec: numerical results}
In this section, we present numerical examples to demonstrate accuracy and stability properties of the FD-DG scheme developed for the elastic wave equation.\footnote{The MATLAB codes used to perform the experiments are available at \url{https://github.com/AndreasGranath/Elastic-Waves-FD-DG}}. First, we examine how the selection of free parameters influences the scheme's stability and time-step restriction. Then, we conduct a series of benchmarking examples to demonstrate the stability and convergence properties of the developed method. This is demonstrated in three scenarios: first, utilizing a manufactured solution in a homogeneous medium with straight boundary and interface; then, introducing a curved boundary within the same medium; and lastly, examining interface waves in a layered medium with discontinuous material property. We end the section with a demonstration of stability by considering a domain with intricate internal geometry, which would be \revone{ cumbersome} to resolve using a curvilinear multiblock FD method. For time stepping, we employ a fourth order accurate modified equation approach, following Gilbert and Joly \cite{Gilbert}.

We discretize in space using the 4th and 6th order SBP-FD operators constructed by Mattsson \cite{Mattsson1}. We view the results in \cite{Almquist2}, where an experimental rate of convergence (EOC) of $4$ was obtained for the $4$th order SBP method and $5$ for a $6$th order method as optimal. This motivates that we use polynomial spaces of order $3$ and $4$ for the DG discretization, respectively. In convergence studies, we choose a time step small enough so that the error in the solution is dominated by the spatial discretization.


\subsection{Spectrum analysis}
 Let the vector $\xx_u$ contain the coordinates of the FD grid on $\Gamma$, and $\xx_{\tilde{u}}$ contain the coordinates of the nodes of the unstructured grid on $\Gamma$. Then, the good projection operators $\Pi^g_{u2\tilde{u}},\Pi^g_{\tilde{u}2u}$  between $\xx_u$ and $\xx_{\tilde{u}}$ are constructed by imposing the accuracy conditions
$$\Pi_{u2\tilde{u}}^g\xx_{u}^k-\xx_{\tilde{u}}^k=\pmb{0},$$
$$\Pi_{\tilde{u}2u}^g\xx_{\tilde{u}}^k-\xx_{u}^k=\pmb{0},$$
 where $\Pi^g_{\tilde{u}2u},\Pi^g_{u2\tilde{u}}$ are assumed to consist of two blocks corresponding to the closure and then a repeated stencil in the interior. Moreover, the power $k$ runs over $k=0,..,q_I$   in the interior of the operator and $k=0,..,q_\partial$ on the closure. Analogous equations are then obtained for $\Pi^b_{u2\tilde{u}}$ and $\Pi^b_{\tilde{u}2u}$ in terms of the good operators using the norm compatibility condition from Definition ~\ref{def:normcompatibility}. The system is subsequently solved using MATLAB's symbolic toolbox, yielding a set of free parameters.
 
 In a setting with exact projection operators, the products $\Pi^g_{\tilde{u}2u}\Pi^b_{u2\tilde{u}}$ and $\Pi^b_{u2\tilde{u}}\Pi^g_{\tilde{u}2u}$ would be identity matrices.  In the work of Kozdon and Wilcox \cite{Kozdon}, the free parameters are utilized such that  the product $\Pi^g_{u2\Tilde{u}}\Pi^b_{\Tilde{u}2u}$  approximates the identity matrix. They achieve this by considering the set $\{\epsilon_i\}_{i=1}^N$ of the sorted eigenvalues belonging to $\Pi^g_{u2\Tilde{u}}\Pi^b_{\Tilde{u}2u}$, and then minimizing the quantity:

\begin{equation}
   \|\Pi^g_{u2\Tilde{u}}\Pi^b_{\Tilde{u}2u}\|_E= \sqrt{\sum_{i=1}^{N-1}(\epsilon_{i+1}-\epsilon_i)^2}.
    \label{eq:kozdonquant}
\end{equation}
We shall expand on this by comparing the function above to another intuitive measure, the \revone{Frobenius} norm of $\Pi^g_{u2\Tilde{u}}\Pi^b_{\Tilde{u}2u}$ , given by

\begin{equation}
\|\Pi^g_{u2\Tilde{u}}\Pi^b_{\Tilde{u}2u}\|_F=\sqrt{\text{Tr}\big((\Pi^g_{u2\Tilde{u}}\Pi^b_{\Tilde{u}2u})^T\Pi^g_{u2\Tilde{u}}\Pi^b_{\Tilde{u}2u}\big)},
\label{eq:FrobNorm}
\end{equation}
as we want the spectrum to be as close to the identity matrix as possible. To optimize accuracy, we minimize the discrete $l^2$-errors to leading order for $\Pi^g_{u2\tilde{u}}$ and $\Pi^b_{\tilde{u}2u}$ respectively. Initially, we minimize the spectrum, and then utilize the resulting value of the parameters as an initial step for subsequent accuracy minimization. However, it is important to note that the order in which we perform these steps is arbitrary. Therefore, it is worthwhile to explore the impact this sequencing has on the overall discretization. We remark that the optimization of free parameters in high order SBP operators was considered in \cite{ErikssonWang}.


To perform the minimization, we use built-in algorithms in MATLAB. We use two derivative free minimization techniques; \textit{fminunc} and \textit{fminsearch} to minimize the quantities \eqref{eq:kozdonquant}-\eqref{eq:FrobNorm}. To investigate the effect of the projection on the overall discretization, we consider the domain $[0,1]^2$ with traction-free boundary conditions, material properties $\mu_{u,v}=1$, $\lambda_{u,v}=1$, $\rho_{u,v}=1$ and an interface at $y_\Gamma=1/5$. The top block is discretized using a fourth order accurate SBP-FD method with $h=1/20$ and the bottom block is discretized by the IPDG method based on cubic basis functions. The triangulation in the bottom block is constructed such that the triangle nodes on the interface match the grid points from the top block. Then, we assemble the bulk matrix of the semidiscretization with parameters   \revthree{$\alpha_u=\beta_u=60$} using the various projection operators obtained from the different measures and minimization techniques. The largest absolute eigenvalues of the bulk matrix,  rescaled by $h^2$, are presented in Table~\ref{table:eigenvals}. The first thing to note is that when the unoptimized projection operators are used, the system becomes very stiff and the CFL condition will render the method pratically unusable, motivating the need of minimization. But then it should be noted that for all different strategies used, we obtain eigenvalues with the same order of magnitude, with the largest discrepancy being when first optimizing spectrum followed by accuracy. In this case, it is clear that the largest difference is roughly a factor of four, which is not significant. We can therefore conclude that all choices above are suitable, but we shall resort to use the operators which resulted in the smallest absolute eigenvalue of $6.104\times 10^3$. 



\begin{table}[h!]
    \centering
    \begin{tabular}{|c|c|c|c|c|}
    \hline
    &\multicolumn{2}{|c|}{Accuracy-Spectrum} & \multicolumn{2}{|c|}{Spectrum-Accuracy}\\
    \hline
        Norm & Fminunc & Fminsearch& Fminunc& Fminsearch \\
        \hline
        $\|P_{u2\Tilde{u}}P_{\Tilde{u}2u}\|_ E$ & $ 6.158\times10^3$& $6.104\times 10^3$ & $2.246\times 10^4$ &$ 6.147\times 10^3 $\\
         $\|P_{u2\Tilde{u}}P_{\Tilde{u}2u}\|_ F$ & $6.119\times 10^3$ & $6.254\times 10^3 $& $6.281\times 10^3$ & $9.380\times 10^3$\\
         \hline
         Unoptimized & \multicolumn{4}{|c|}{$1.783\times 10^8$} \\
         \hline
    \end{tabular}
    \caption{The magnitude of the largest eigenvalues of the bulk matrix for the semidiscretization rescaled by $h^2$ depending on the norms, order in which we optimize and the method used.}
    \label{table:eigenvals}
\end{table}

The Cartesian grid in the finite difference method is constructed independently from the unstructured grid of the DG method. Thus, on the interface, a grid point from the Cartesian grid may be arbitrarily close to a node on the unstructured grid, illustrated in \ref{fig:shiftednodes}. In this case, we have an arbitrarily small interval on the glue grid. In the following, we investigate how this affects the projection operators and the overall discretization. More specifically, we use the same geometric setup and material parameters in the previous example. On the interface, the grid points from the Cartesian grids coincide with the nodes from the unstructured grid, except in one location where the distance is $\delta h$ where $\delta\in\{10^{-6},5\times10^{-6},10^{-5},\cdot\cdot\cdot,5\times 10^{-2},10^{-1}\}$.
We compute the largest absolute eigenvalue of the overall discretization matrix. The results are shown in Figure~\ref{fig:maxeigs} for four different grid resolutions. From the scale of the eigenvalues, it is clear that the method is robust to non-matching FD and DG grids as the eigenvalue remains essentially constant for a fixed grid.
\\
\begin{figure}[h!]
    \centering
    \begin{subfigure}[t]{0.47\linewidth}
    \centering
        \includegraphics[width=\textwidth]{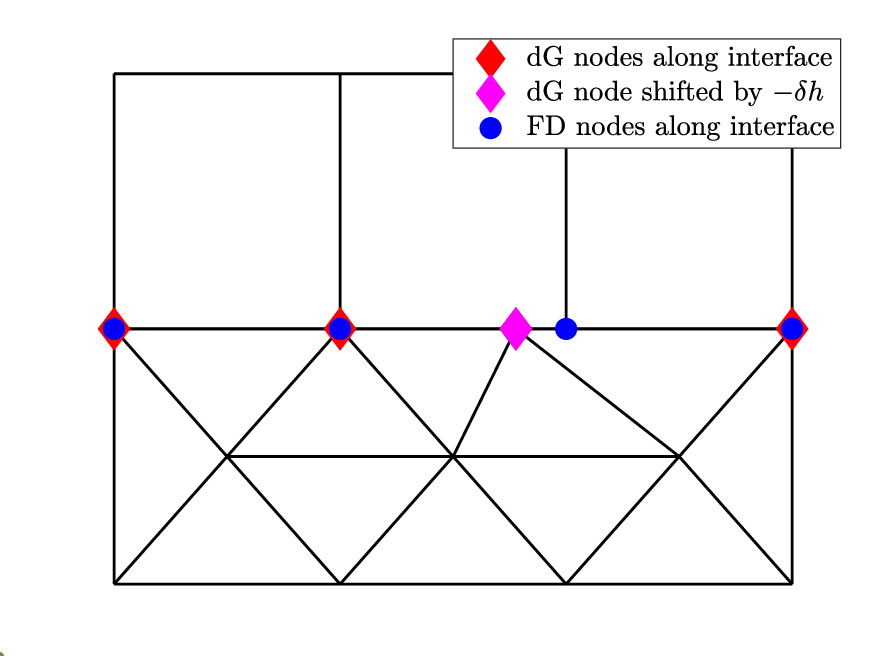}
        \subcaption{}
        \label{fig:shiftednodes}
    \end{subfigure}
       \begin{subfigure}[t]{0.47\linewidth}
        \centering
        \includegraphics[width=\textwidth]{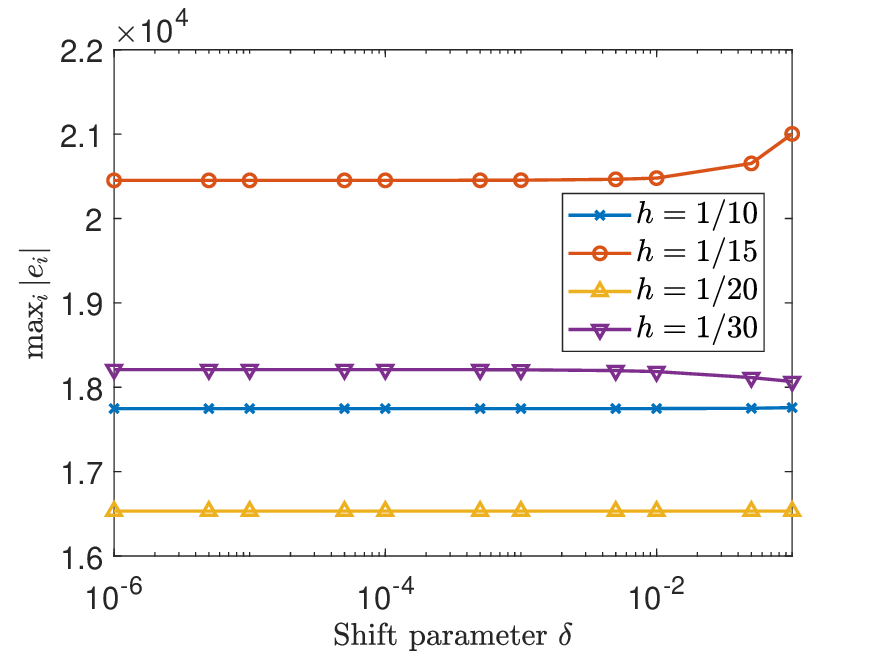}
        \subcaption{}
         \label{fig:maxiegs}
    \end{subfigure}
    \centering
    \caption{(a) An illustration of the grid on the interface (b) The largest absolute eigenvalues $\max_i|e_i|$ of the discretization matrix rescaled by $h^2$ as a function of $\delta$. Note the small interval along the $y$-axis.}
    \label{fig:maxeigs}
\end{figure}

\subsection{Convergence rate}
\subsubsection{Continuous materials}
We again consider domain $\Omega=[0,1]^2$ with all material parameters chosen to be equal, $\lambda_u=\mu_u=\lambda_v=\mu_v=1$ and with unit densities $\rho_u=\rho_v=1$. \revthree{Analogously, the stability parameters are chosen to be equal in both domains and given by $\alpha=\beta=60$.} Then, we shall investigate the accuracy of the scheme by using the method of manufactured solution when the FD grid points match the triangle nodes of the IPDG on the interface. Let $N$ denote the number of FD grid points on the interface. We note that even in this case, the projection operators are not identity operators because of different discrete norm in the FD and DG discretization. The solution is chosen to be given by the two fields
$$w_1=\sin(5x+7y-t)\quad\text{ and}\quad w_2=\cos(7x+5y-2t),$$
in both the FD and DG domain. We impose traction boundary conditions, and compute the boundary data and the forcing function by using the manufactured solution.

  We shall carry out two experiments in this subsection, where the error in both are calculated as $||\ww-\ww_h||_h$ where $\ww=[\uu,\vv]^T$ and the norm is given by $||\ww||_h^2=\uu^T\HH\uu+(\vv,\vv)_{\Omega_v}$. In both of the cases, the wave evolves until the final time $T=1$.

In the first experiment, we partition the domain according to $\Omega=\Omega_u\cup\Omega_v$ with $\Omega_u=[0,1]\times[1/5,1]$ and $\Omega_v=[0,1]\times[0,1/5]$ and perform a convergence study when refining the mesh by increasing $N$ by increments of ten, starting at $N=21$. An illustration of the coarsest mesh is given in Fig.~\ref{fig:firstdomain} and the resulting discrete errors in $\|\cdot\|_h$-norm are presented in Fig.~\ref{fig:trigsolnaligned}. Moreover, we present the convergence rate in Table~\ref{table:results2}. From the table, it is clear that the scheme demonstrates a convergence rate of four for the fourth order SBP operator and five for the sixth order one. This follows the optimal convergence trend $\min\{2p,p+2\}$ for an FD method of order $2p$, which also is the case for the FD-DG discretization in second-order form of the scalar wave equation \cite{Wang1}. 

\begin{table}[h!]
    \centering
    \begin{tabular}{|c|c|c|c|c|}
    \hline
      $N$   & Error $p=2$ & EOC $p=2$ & Error $p=3$ & EOC $p=3$\\
      \hline
        $21$ & $4.408\times 10^{-3}$ & - & $5.691\times 10^{-3}$ & -  \\
         $31$ &$9.373\times 10^{-4}$ &$3.818$ & $6.917\times 10^{-4}$ & $5.198$ \\
         $41$ &$3.001\times 10^{-4}$& $3.958$&$1.597\times 10^{-4}$& $5.095$\\
         $51$ & $1.191\times 10^{-4}$& $4.139$&$5.130\times 10^{-5}$& $5.089$\\
         $61$ &$5.633\times 10^{-5}$ &$4.109$ &$2.042\times 10^{-5}$&$5.052$\\
         $71$ &$2.989\times 10^{-5}$ &$4.104$&$9.368\times 10^{-6}$& $5.055$\\
         \hline
    \end{tabular}
    \caption{Discrete errors in $||\cdot||_h$-norm with associated EOCs at final time $T=1$ in the case of a continuous material.}
    \label{table:results2}
\end{table}

\begin{figure}[h]
    \centering
    \begin{subfigure}[t]{0.47\linewidth}
    \centering
        \includegraphics[width=\textwidth]{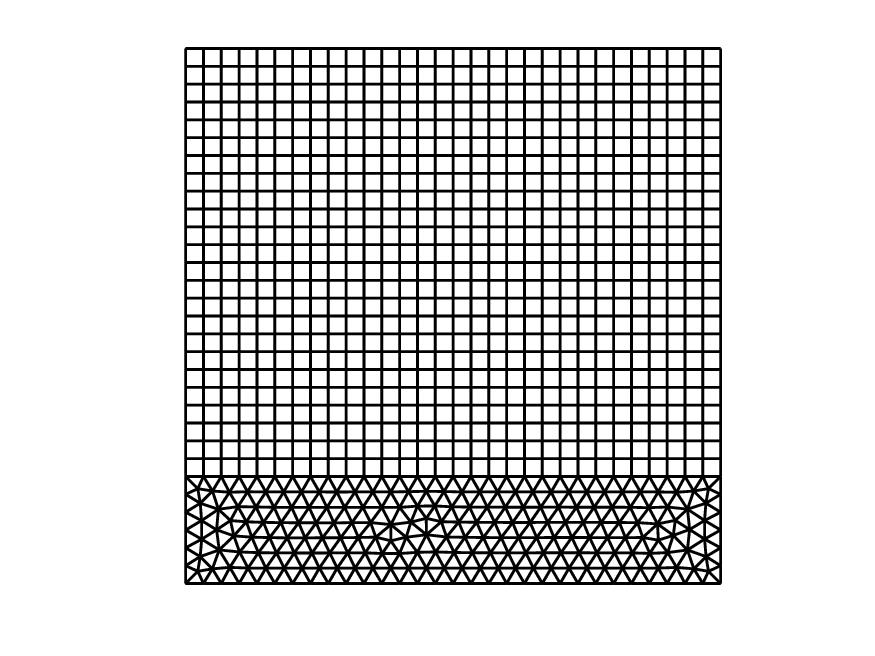}
        \subcaption{}
        \label{fig:firstdomain}
    \end{subfigure}
       \begin{subfigure}[t]{0.47\linewidth}
        \centering
        \includegraphics[width=\textwidth]{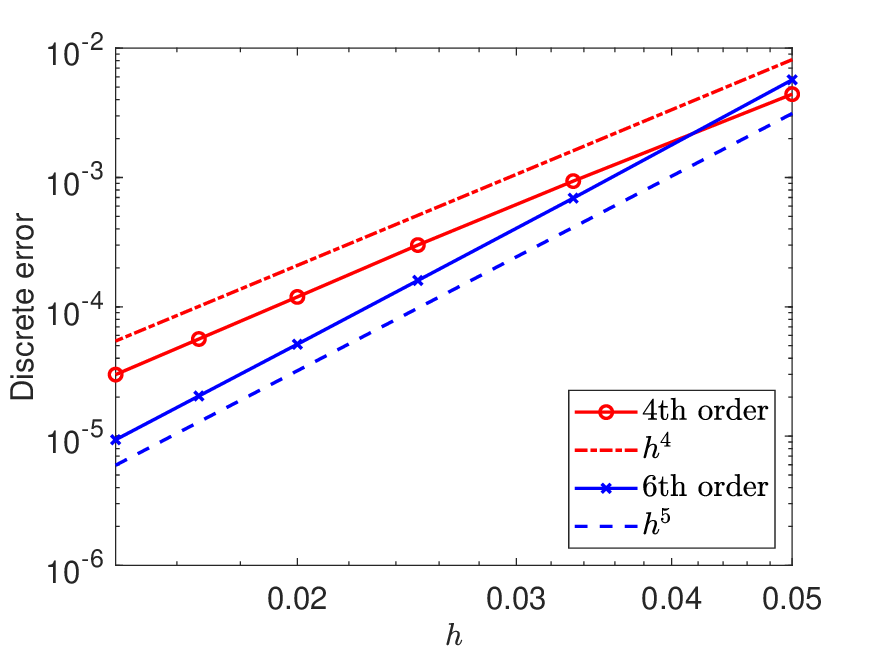}
        \subcaption{}
         \label{fig:trigsolnaligned}
    \end{subfigure}
    \small{\caption{(a) The coarsest discretization of the domain used with matching grids (b) The discrete error in $||\cdot||_h$-norm at the final time $T=1$.}}
\end{figure}

\revthree{When using a DG scheme based on polynomial degree higher than one, there are interior nodes on the triangles. We now perform the same experiment but on a mesh where the FD grid points match all nodes on the triangle edges on the interface, see an illustration in Fig.~\ref{fig:aligningdomain}. The discrete errors in $\|\cdot\|_h$-norm are presented in Table~\ref{table:aligningnodesresult} and Fig.~\ref{fig:aligningnodes}. We observe that the scheme maintains the expected convergence rate of $\min\{2p,p+2\}$ even on very fine meshes. We have not observed a clear loss of a half order in convergence rate as one would expect to happen for the SBP-FD discretization, demonstrated on a curvilinear mesh in \cite{Almquist1}.}

\begin{figure}[h]
    \centering
    \begin{subfigure}[t]{0.47\linewidth}
    \centering
        \includegraphics[width=\textwidth]{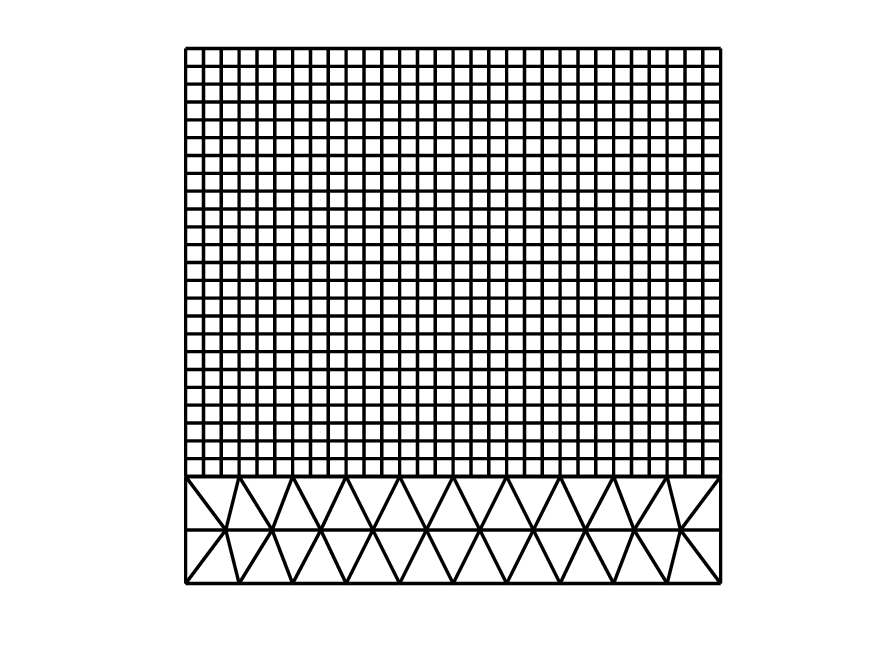}
        \subcaption{}
        \label{fig:aligningdomain}
    \end{subfigure}
       \begin{subfigure}[t]{0.47\linewidth}
        \centering
        \includegraphics[width=\textwidth]{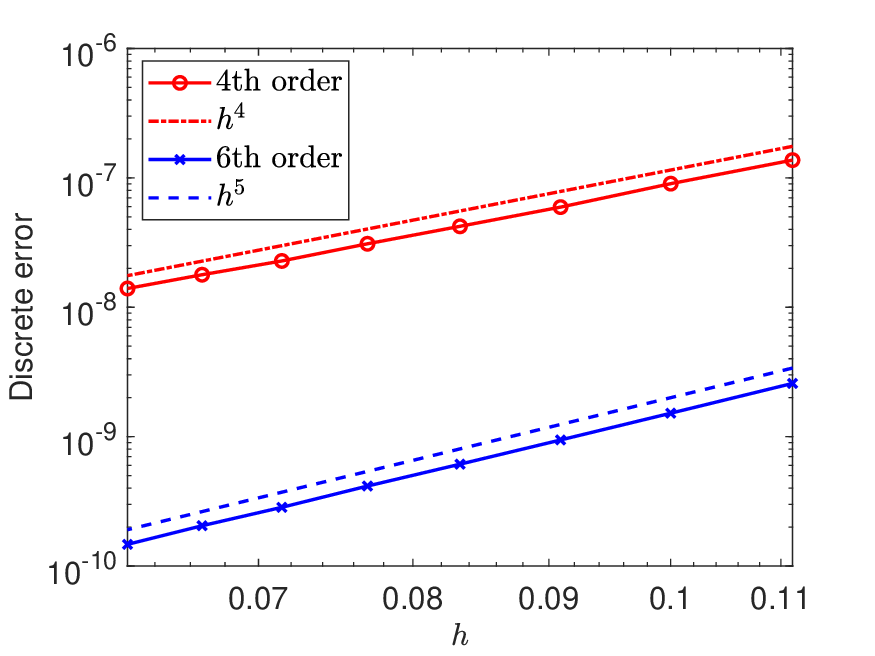}
        \subcaption{}
         \label{fig:aligningnodes}
    \end{subfigure}
    \small{\caption{(a) An illustration of the mesh structure used for mathing degrees of freedom with $p=3$, corresponding to $N=31$. (b) Numerical errors in $\|\cdot\|_h$ norm of the fourth and sixth order experiments for matching degrees of freedom along the interface.}}
\end{figure}

\begin{table}[h!]
    \centering
    \begin{tabular}{|c|c|c|c|c|c|}
    \hline
      $N$   & Error $p=2$ & EOC $p=2$ &$N$& Error $p=3$ & EOC $p=3$\\
      \hline
         $271$ &$1.372\times 10^{-7}$ &$-$ & $361$ & $2.578\times 10^{-9}$ & $-$ \\
         $301$ &$9.020\times 10^{-8}$& $3.980$&$401$ &$1.517\times 10^{-9}$& $5.031$\\
         $331$ & $5.977\times 10^{-8}$& $4.317$&$441$&$9.422\times 10^{-10}$& $5.002$\\
         $361$ &$4.223\times 10^{-8}$ &$3.992$ &$481$&$6.122\times 10^{-10}$&$4.881$\\
         $391$ &$3.094\times 10^{-8}$ &$3.886$&$521$&$4.149\times 10^{-10}$& $4.939$\\
         $421$ &$2.279\times 10^{-8}$ & $4.125$ &$561$&$2.841\times 10^{-10}$ & $5.110$ \\
         $451$ & $1.784\times 10^{-8}$ & $3.554$ &$601$& $2.048\times 10^{-10}$ & $4.745$ \\
         $481$ & $1.396\times 10^{-8}$ & $3.791$ &$641$& $1.469\times 10^{-10}$& $5.147$ \\
         \hline
         \multicolumn{2}{|c|}{Least-square EOC}  &3.985 & \multicolumn{2}{|c|}{Least-square  EOC} & 4.965\\
         \hline
    \end{tabular}
    \caption{Discrete errors in $||\cdot||_h$-norm with associated EOCs at final time $T=1$ in the case when the FD and DG degrees of freedom align. Here, $N$ denotes the number of FD grid points on the interface.}
    \label{table:aligningnodesresult}
\end{table}

\revthree{Finally, we show in Table~\ref{table:DOFs}  the number of degrees of freedom (DOFs) in the FD and DG schemes with $N=71$ FD grid points on the interface, which match the triangle nodes from the IPDG discretization. We also present the number of nonzero entries in the corresponding discretization matrices when $p=2$. From this, it is clear that the number of nonzeros in the discretization matrix for the DG domain is much larger than that for the FD domain, despite being a quarter of the size of the FD domain.}

\begin{table}[h!]
    \centering
    \begin{tabular}{|c|c|c|c|c|}
    \hline
      $N$   & FD DOFs & FD nonzeros & DG DOFs & DG nonzeros\\
      \hline
        $71$ & $5041$ & $247978$ & $22440$ & $3523200$  \\
        
         \hline
    \end{tabular}
    \caption{Number of DOFs in the FD and DG domains together with the number of non-zero entries in the corresponding discretization matrices.}
    \label{table:DOFs}
\end{table}

In our last experiment of this subsection, we shall instead let the overall domain $\Omega$ have its bottom boundary curved, parametrized by $f_s(x)=0.4x(x-1)e^{-10(x-0.5)^2}$ and partition it as in the previous example $\Omega_u=[0,1]\times[1,1/5]\times[1/5,y_s]$ , where $y_s$ denotes the function $f_s$ evaluated along $[0,1]$. Moreover, we use the same material parameters and densities as in the previous experiment, but let the solution be given by the fields

$$w_1=\sin(3x+2y-t)\quad\text{ and }\quad w_2=\sin(2x+3y-2t)$$
in both FD and DG domains. Again we let the boundary traction and forcing be determined by the chosen manufactured solution. The resulting fourth and sixth order errors are presented together with the convergence orders in Table \ref{table:curvedbdryresults} and the overall behavior is illustrated in Fig.~\ref{fig:curvedbdryresults}. We again note that both the fourth and sixth order methods result in an optimal convergence order even in the presence of a curved boundary.

It should be noted that we have used straight-edged triangles throughout the entire domain, i.e. the curved boundary is approximated by line segments, and the traction is also calculated on these segments, as illustrated on a coarse mesh in Fig.~\ref{fig:curvedbdrydomain}. This approximation of the traction is necessary to obtain an overall consistent method. When instead using the exact normal to obtain the boundary traction, a loss of two orders of accuracy was observed. This can be addressed by using triangles with curved sides close to the boundary. 

\begin{table}[h!]
    \centering
    \begin{tabular}{|c|c|c|c|c|}

    \hline
      $N$   & Error $p=2$ & EOC $p=2$ & Error $p=3$ & EOC $p=3$  \\
      \hline
        $21$ & $2.936\times 10^{-4}$ & - & $1.215\times 10^{-4}$& - \\   
         $31$ & $5.771\times 10^{-5}$&$4.012$& $1.509\times 10^{-5}$& $5.144$  \\
         $41$ &$1.823\times 10^{-5}$& $4.004$& $3.413\times 10^{-6}$&$5.168$\\
         $51$ &$7.898\times 10^{-6}$ &$3.750$& $1.08\times 10^{-6}$&$5.134$ \\
         $61$ &$3.796\times 10^{-6}$&$4.018$ & $4.446\times 10^{-7}$ &$4.895$\\
         $71$ &$2.0307\times 10^{-6}$&$4.058$ & $2.014\times 10^{-7}$ &$5.137$\\
         \hline
    \end{tabular}
    \caption{Discrete errors in $||\cdot||_h$-norm and corresponding EOCs at final time $T=1$ when the material is continuous and the bottom boundary is curved.}
    \label{table:curvedbdryresults}
\end{table}

\begin{figure}[h!]

      \centering
    \begin{subfigure}[t]{0.47\linewidth}
    \centering
        \includegraphics[width=\textwidth]{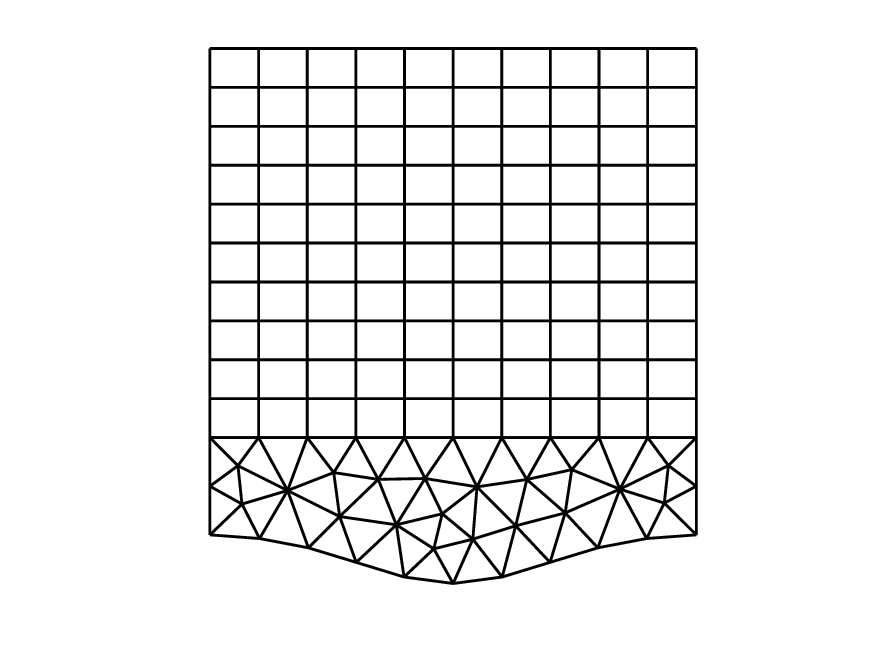}
        \subcaption{}
        \label{fig:curvedbdrydomain}
    \end{subfigure}
       \begin{subfigure}[t]{0.47\linewidth}
        \centering
        \includegraphics[width=\textwidth]{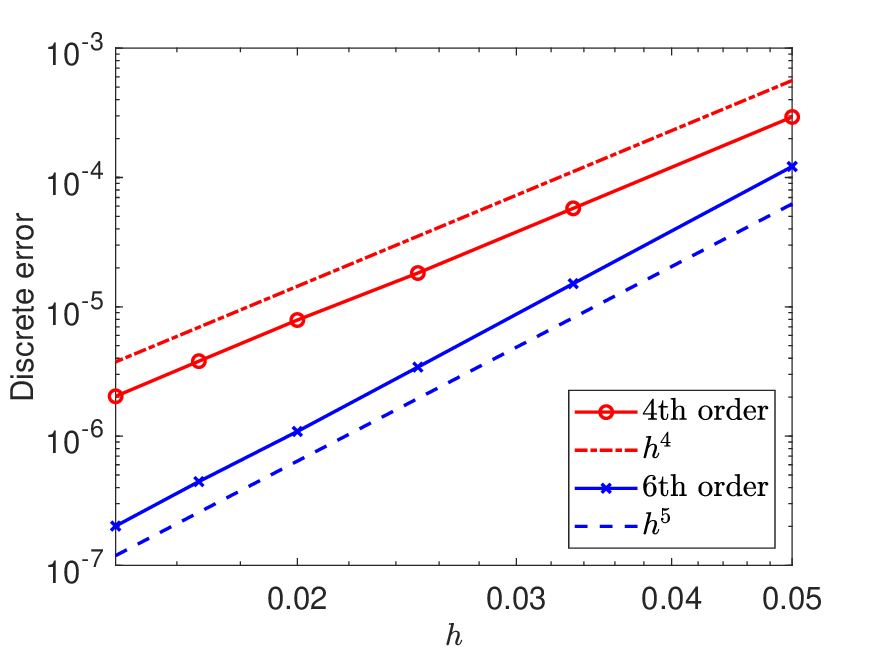}
        \subcaption{}
         \label{fig:curvedbdryresults}
    \end{subfigure}
       \caption{(a) A coarse discretization of the domain demonstrating the linear approximation of the southern boundary given by $f_s(x)$ (b) The discrete error at final time $T=1$ in $||\cdot||_h$-norm when the material is continuous and the bottom boundary of the DG domain is curvilinear, given by $f_s(x)=0.4x(x-1)e^{-10(x-0.5)^2}$ for both the fourth- and sixth-order methods.}
\end{figure}

\subsubsection{Interface waves in discontinuous materials}
If the domain $\Omega$ consists of two subdomains $\Omega_u$ and $\Omega_v$ having different material parameters, we can have waves localizing to the interface $\Gamma$ between the two. A common example of these are so-called \textit{Stoneley waves} which propagate periodically in the $x$-direction with a velocity $c_s$ and decay exponentially from the interface. \revone{ The components of the Stoneley waves propagate with  pressure and shear wave velocities
$$\alpha_u=\sqrt{\frac{2\mu_u+\lambda_u}{\rho_u}},\quad \beta_u=\sqrt{\frac{\mu_u}{\rho_u}},\quad \alpha_v=\sqrt{\frac{2\mu_v+\lambda_v}{\rho_v}},\quad \beta_v=\sqrt{\frac{\mu_v}{\rho_v}},$$}
\revone{and wave number $\xi>0$} . Assuming that $\Omega_u=[0,1]\times[y_\Gamma,1]$ and $\Omega_v=[0,1]\times[0,y_\Gamma]$ for some suitable $0<y_\Gamma<1$ we can write the waves in $\Omega_u$ on the following explicit form
\begin{align}
    \begin{bmatrix}
        u_1(\cdot,t)\\
        u_2(\cdot,t)
    \end{bmatrix}&=Ae^{-\gamma_u(y-y_\Gamma)}\begin{bmatrix}
        \xi\cos(\xi(x-c_st)) \\
        -\gamma_u\sin(\xi(x-c_st))
    \end{bmatrix}\nonumber\\
    &+Be^{-\eta_u(y-y_\Gamma)}\begin{bmatrix}
        -\eta_u\cos(\xi(x-c_st))\\
        \xi\sin(\xi(x-c_st))
    \end{bmatrix},
    \label{eq:stoneley_U}
\end{align}
where \revone{$\gamma_u=\frac{\xi}{\alpha_u}\sqrt{\alpha_u^2-c_s^2},\eta_u=\frac{\xi}{\beta_u}\sqrt{\beta_u^2-c_s^2}$} and in the domain $\Omega_v$ we instead have
\begin{align}
    \begin{bmatrix}
        v_1(\cdot,t)\\
        v_2(\cdot,t)
    \end{bmatrix}&=Ce^{\gamma_u(y-y_\Gamma)}\begin{bmatrix}
        \xi\cos(\xi(x-c_st)) \\
        \gamma_v\sin(\xi(x-c_st))
    \end{bmatrix}\nonumber\\
    &+De^{\eta_v(y-y_\Gamma)}\begin{bmatrix}
        \eta_v\cos(\xi(x-c_st))\\
        \xi\sin(\xi(x-c_st))
    \end{bmatrix}
    \label{eq:Stoneley_V}
\end{align}
with \revone{ $\gamma_v=\frac{\xi}{\alpha_v}\sqrt{\alpha_v^2-c_s^2}$ and $\eta_v=\frac{\xi}{\beta_v}\sqrt{\beta_v^2-c_s^2}$}. To determine the coefficients $A,B,C$ and $D$ and $c_s$ we substitute \eqref{eq:stoneley_U} and \eqref{eq:Stoneley_V} into the interface conditions \eqref{eqn: contcond}-\eqref{eqn: traccond} obtaining a linear system
\begin{equation}
    \mathcal{K}[A,B,C,D]^T=[0,0,0,0]^T,
    \label{eq:StoneleySystem}
\end{equation}
where the matrix $\mathcal{K}$ is given by
{\footnotesize
\begin{equation*}
    \mathcal{K}=\begin{bmatrix}
        \xi && -\eta_u && -\xi && -\eta_v\\
        -\gamma_u && \xi && -\gamma_v && -\xi\\
        -2\mu_u\gamma_u && \mu_u(\eta_u^2+\xi^2) && -2\mu_v\gamma_v\xi && -\mu_v(\eta_v^2+\xi^2)\\
        (2\mu_u+\lambda_u)\gamma_u^2-\xi^2\lambda_u && -2\mu_u\eta_u\xi && -(2\mu_v+\lambda_v)\gamma_u^2+\xi^2\lambda_v && -2\mu_v\eta_v\xi
    \end{bmatrix},
\end{equation*}
}
which depends implicitly on the wave speed $c_s$. In order for \eqref{eq:StoneleySystem} to have a unique solution, we fix the wave speed $c_s$ by solving $\det{\mathcal{K}(c_s)}=0$ . Then it should be noted that the equation is underdetermined, which allows us to set $A=1$ and then solve for $B,C$ and $D$ using the previously determined $c_s$. For a deeper discussion on the theory for this type of waves and their discretizations, see \cite{Graff}. 

In this experiment, we use the same computational domain as previously, $\Omega_u=[0,1]\times[1/5,1]$, $\Omega_v=[0,1]\times[0,1/5]$, fix the wavelength to $1$ by setting $\xi=2\pi$ \revthree{and let the stability parameters be given by $\alpha=\beta=150/c_s$}. The nodes along the FD-DG interface are chosen to match and be given by $N$ as in the previous experiments. The material parameters used, together with the resulting amplitudes $B,C$,$D$ and wave speed $c_s$ are presented in Table~\ref{table:stoneleyprops}.
\begin{table}[h!]
    \centering
    \begin{tabular}{|c|c|c|c|c|c|}
        \hline
       $\Omega_u$ & Value & $\Omega_v$ & Value & Amplitudes & Value\\
        \hline
         $\mu_u$& $0.1$ & $\mu_v$ & $0.2$ & $B$ & $-3.417353225815728$ \\
         $\lambda_u$ & $1$ & $\lambda_v$ & $3$ & $C$ & $1.792431528662961$\\
         $\rho_u$ & $1$ & $\rho_v$ & $1.9999$ & $D$ & $ -6.114370209408487$\\
         \hline
         $c_s$ & \multicolumn{5}{|c|}{$0.315138096754869$}\\
         \hline
        
    \end{tabular}
    \caption{Parameters used in the Stoneley wave convergence study with the resulting wave speed $c_s$ and amplitudes when fixing $A=1$.}
    \label{table:stoneleyprops}
\end{table}
Letting the boundary traction and forcingr follow from the manufactured solutions \eqref{eq:stoneley_U}-\eqref{eq:Stoneley_V}, we obtain the errors and convergence orders presented in Table~\ref{table:StoneleyResults}, with the overall behaviour illustrated in Figure~\ref{fig:StoneleyErrors}. The resulting fields at the first time-step are illustrated in Figure~\ref{fig:StoneleyField1}-\ref{fig:StoneleyField2}. While the convergence orders follow the same trend as in previous examples we note that the overall errors are larger in magnitude. This should at least partly be attributed to the large scaling used in the manufactured solution. 

\begin{table}[h!]
    \centering
    \begin{tabular}{|c|c|c|c|c|}
   
    \hline
      $N$   & Error $p=2$ & EOC $p=2$ & Error $p=3$ & EOC $p=3$ \\
      \hline
        $81$ & $5.587\times 10^{-4}$ &$ - $&$1.870\times 10^{-4}$ & - \\
         $91$ & $3.397\times 10^{-4}$ &$4.222$  &$1.041\times 10^{-4}$ & $4.968$ \\
         $101$ & $2.232\times 10^{-4}$ & $3.988$& $6.173\times 10^{-5}$  & $4.967$\\
         $111$ & $1.538\times 10^{-4}$&$3.906$ & $3.838\times 10^{-5}$ & $4.986$ \\
         $121$ & $1.084\times 10^{-4}$&$4.021$&$2.488\times 10^{-5}$ & $4.979$ \\
         $131$ & $7.846\times 10^{-5}$&$4.037$& $1.670\times 10^{-5}$ & $4.978$ \\
         \hline
    \end{tabular}
    \caption{Convergence rates and discrete errors in $||\cdot||_h$-norm for the Stoneley wave problem at time $T=1$ for the fourth and sixth order methods.}
    \label{table:StoneleyResults}
\end{table}

\begin{figure}[h!]
    \centering
    \begin{subfigure}[t]{0.225\textwidth}
    \centering
        \includegraphics[width=\textwidth]{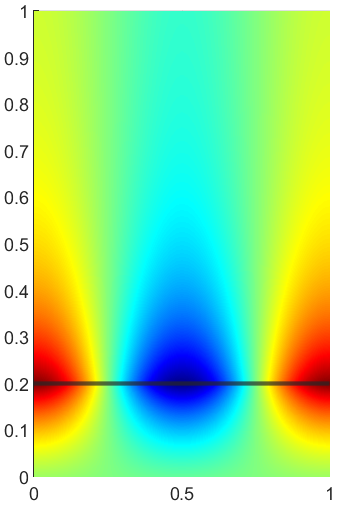} 
        \subcaption{}
          \label{fig:StoneleyField1}
    \end{subfigure}
      \begin{subfigure}[t]{0.225\textwidth}
    \centering
        \includegraphics[width=\textwidth]{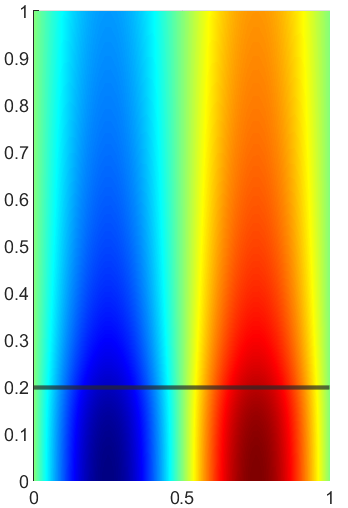} 
        \subcaption{}
          \label{fig:StoneleyField2}
    \end{subfigure}
       \begin{subfigure}[t]{0.45\textwidth}
        \centering
        \includegraphics[width=\textwidth]{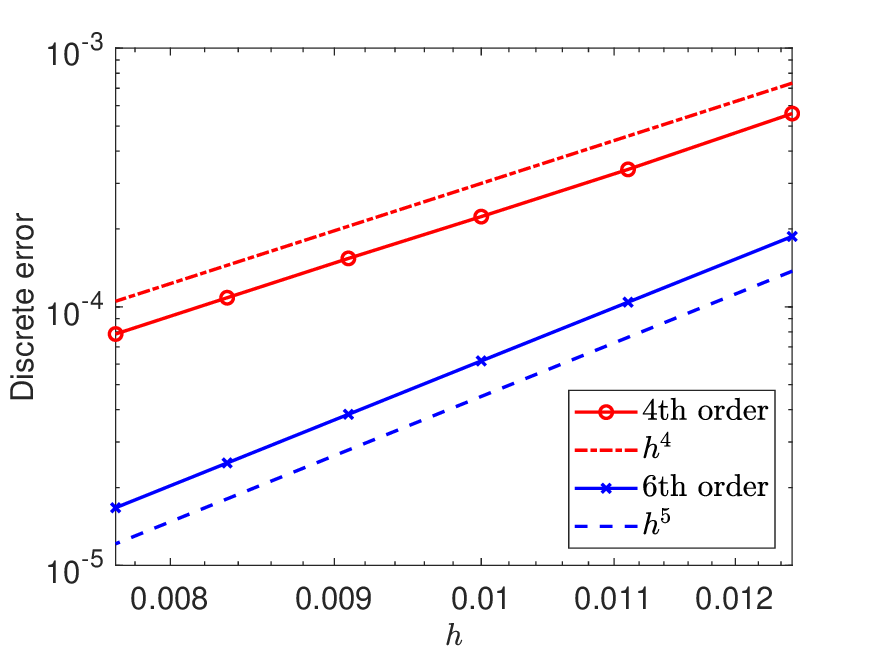} 
        \subcaption{}
        \label{fig:StoneleyErrors}
        \end{subfigure}

    \caption{(a)-(b): Illustration of the resulting horizontal and vertical fields $[u_1,v_1]^T, [u_2,v_2]^T$ at the first time-step (c): Numerical errors in $||\cdot||_h$-norm of the Stoneley waves for fourth and sixth order experiments.}
\end{figure}

\subsection{Complex geometry}
Finally, to demonstrate stability we consider a layered domain $\Omega=[0,1]\times[0,2]$ containing a complex geometric feature. The two layers are given by $\Omega_u=[0,1]\times[0.2,2]$, $\Omega_v=[0,1]\times[0,0.2]$, where the geometric feature is in $\Omega_v$. Moreover, we let the material parameters in the two domains be given by $\mu_u=0.1, \lambda_u=1$, $\mu_v=0.01$, $\lambda_v=0.1$ and densities $\rho_u=\rho_v=1$.  This motivates using an unstructured method in $\Omega_v$.  We let the two meshes along the interface be conforming and use the same penalization parameters $\alpha=\beta=60$ in both domains. The fields are then initialized at $t=0$ as two symmetric Gaussian pulses given by
\begin{align}
    w(x,y)&=550e^{-1000[(x-0.4)^2+(y-0.135)^2]},\nonumber\\
    &+550e^{-1000[(x-0.6)^2+(y-0.135)^2]}.
    \label{eq:GaussianPulse}
\end{align}
The pules then propagate freely without any external load, imposing traction free conditions on the domain boundaries as well as the internal cavities. The resulting distribution of the magnitude of the pulse at various time steps are shown in Figure~\ref{fig:pulse1}-\ref{fig:pulse4}.

In Figure~\ref{fig:pulse1}, we see the internal geometry in $\Omega_v$ clearly. \revone{Resolving the geometry using a curvilinear FD grid would be impractical as it would require us to partition the domain into a large number of blocks to resolve the arcs and corners properly.} Moreover, we see the effect of the pulse propagating from $\Omega_v$ into $\Omega_u$ continuously. In Figure~\ref{fig:pulse2}, we can clearly see the shift in wavelength of the solution as it enters a domain with larger wave speed, together with the reflections arising at the interface. Figure~\ref{fig:pulse3} then demonstrates how the wave has interacted with the top boundary and produce reflections close to the interface. Finally at $T=100$ in Figure~\ref{fig:pulse4},  we observe  that the discretization remains stable.

\begin{figure}[h]
    \centering
    \begin{subfigure}[t]{0.24\textwidth}
    \centering
        \includegraphics[width=\textwidth]{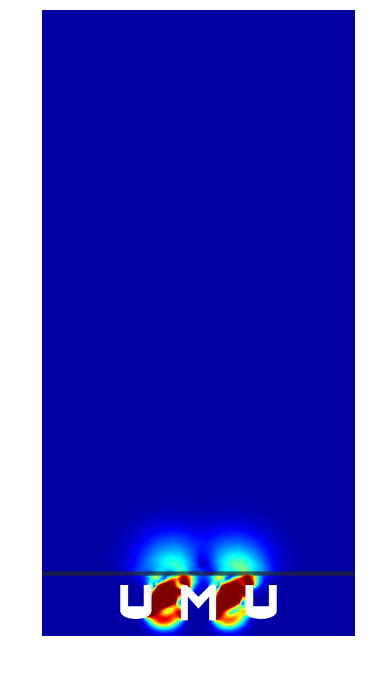} 
        \subcaption{}
          \label{fig:pulse1}
    \end{subfigure}
       \begin{subfigure}[t]{0.24\textwidth}
        \centering
        \includegraphics[width=\textwidth]{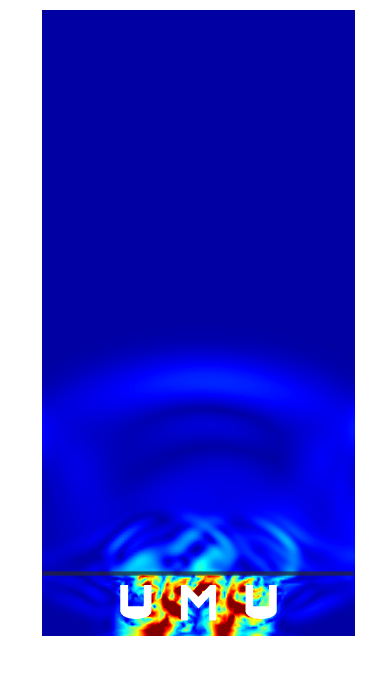}      
      \subcaption{}
         \label{fig:pulse2}
         \label{fig:pulse2}
    \end{subfigure}
       \begin{subfigure}[t]{0.24\textwidth}
        \centering
        \includegraphics[width=\textwidth]{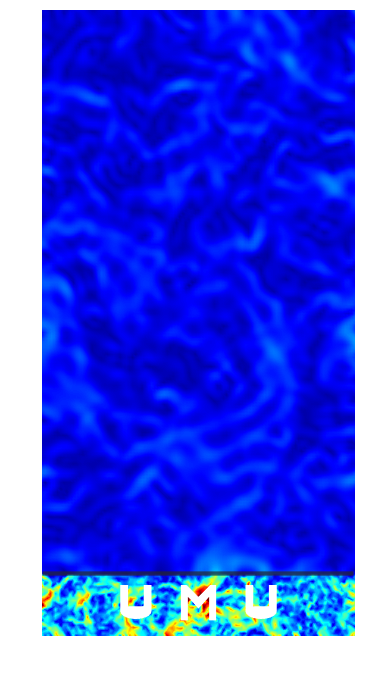}
        \subcaption{}
         \label{fig:pulse3}
    \end{subfigure}
           \begin{subfigure}[t]{0.24\textwidth}
        \centering
        \includegraphics[width=\textwidth]{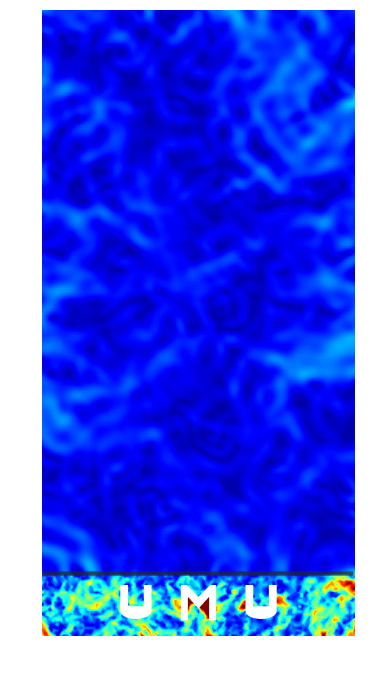}       
        \subcaption{}
         \label{fig:pulse4}
    \end{subfigure}
    \caption{The distribution of the magnitude of the Gaussian pulse given by \eqref{eq:GaussianPulse} at the time steps (a) $T=0.2202$ (b) $T=0.7202$ (c) $T=50$ and (d) $T=100$.}
\end{figure}
To end this section, we perform seismic surveys in $\Omega$ at  points $(x_{S_a},y_{S_a})=(0.1,1.8)$ and $(x_{S_b},y_{S_b})=(0.9,0.4)$ using the same setup as in the previous example. We choose the number of grid points along the interface as $N\in\{61,81,121,161\}$, with the expectation that the profile of the pulse should converge as the mesh-size decreases. The pulse is measured until $T=3$. In Figure~\ref{fig:monitor1}, we show the profile recorded at $(x_{S_a},y_{S_a})=(0.1,1.8)$. This point is far from the interface and the wave has not reached here until $t=1.5$. The other point  $(x_{S_b},y_{S_b})=(0.9,0.4)$ is close to the interface, and the recorded profile is shown in Figure~\ref{fig:monitor2}. In both case, we observe a clear convergence when decreasing the mesh size $h$.

\begin{figure}[h!]
    \centering
    \begin{subfigure}[t]{0.49\textwidth}
    \centering
        \includegraphics[width=\textwidth]{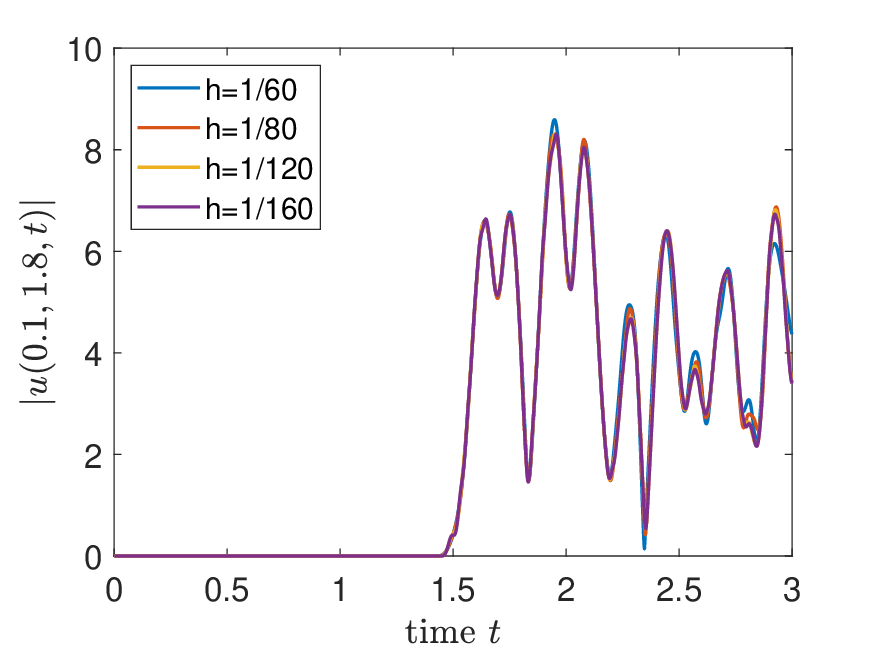} 
        \subcaption{}
          \label{fig:monitor1}
    \end{subfigure}
       \begin{subfigure}[t]{0.49\textwidth}
        \centering
        \includegraphics[width=\textwidth]{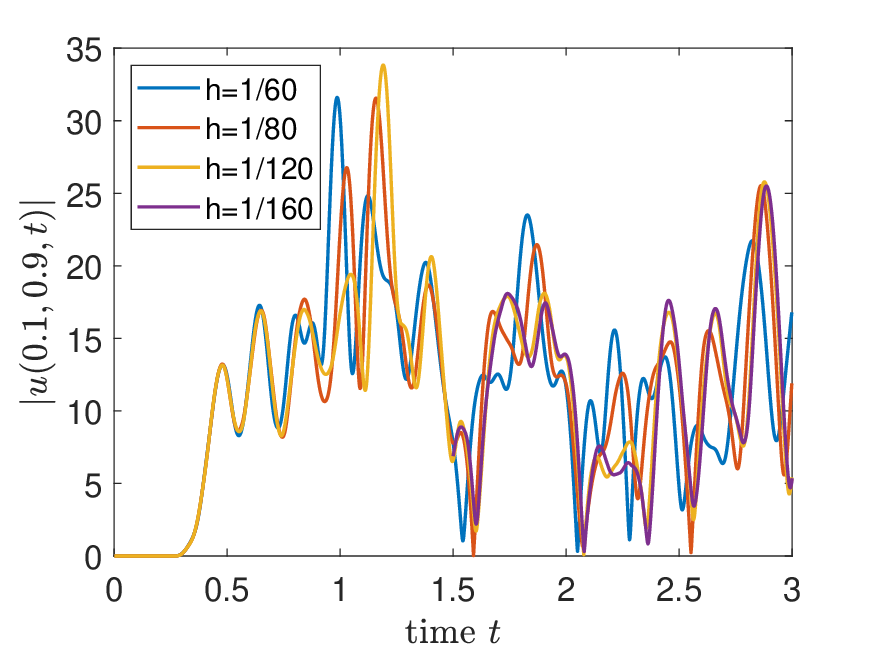} 
        \subcaption{}
        \label{fig:monitor2}
        \end{subfigure}

    \caption{The magnitude of the pulse as observed during the time interval $t\in[0,3]$ at the survey point  (a) $(x_{S_a},y_{S_a})=(0.1,1.8)$  (b)  $(x_{S_b},y_{S_b})=(0.9,0.4)$.}
    \label{fig:surveys1}
\end{figure}

\section{Conclusion}
\label{sec: conclusion}
In this paper, we have constructed a provably stable hybrid FD-DG method for solving the elastic wave equations in second-order form. Our approach uses the computationally efficient SBP finite difference method on Cartesian grids across the majority of the domain, particularly away from complex geometric features. In regions close to the complex geometry, we instead employ the  IPDG method on unstructured grid, known for its ability to resolve such complexities more efficiently. The coupling between these methods is carried out through a penalty technique using norm-compatible projection operators.

Furthermore, we have examined the influence of the projection operators on the spectrum of the overall discretization. We have found that setting all free parameters obtained during the assembly of the projection operators to zero yields a very small CFL number, rendering the method practically unusable. We have investigated different strategies to optimize the free parameters for accuracy and spectrum. Our method has also demonstrated robustness when projecting between grids that are nearly matching.

\section*{Acknowledgement}
A.G is thankful for the economical support in the form of a travel grant from the SJCKMS foundation.






\appendix

\section{Proof of positive definiteness}
\label{section:Appendix}
Here we prove that the matrix \eqref{eq:u2energy} is positive semi-define under suitable choice of $\beta_u$. Recall the matrix $\ZZ=[\ZZ_{11},\ZZ_{12} ; \ZZ_{21},\ZZ_{22}]$, then by using the definition of the material matrix $\LL_3$ we see that $\ZZ_{22}$ becomes
\begin{align*}
    \ZZ_{22}&=\begin{bmatrix}
        \mu [C'-\frac{\mu_u}{4\beta_u\chi_2}](H_x\otimes I_y) && 0\\
        0 && (2\mu_u+\lambda_u)[C'-\frac{(2\mu_u+\lambda_u)}{4\beta_u\chi_2}](H_x\otimes I_y).
    \end{bmatrix}
\end{align*}
Since the norm matrix $H_x$ is diagonal with strictly positive eigenvalues we see that $\ZZ_{22}$is invertible and positive semi-definite if $\beta_u\geq\max\bigg\{\frac{2\mu_u+\lambda_u}{4C'\chi_2},\frac{\mu_u}{4C'\chi_2}\bigg\}$ since the eigenvalues necessarily are real. From the assumption that $\mu_u+\lambda_u>0$ it must then hold that the second term is smaller and hence we must have the bound $\beta_u\geq\frac{2\mu_u+\lambda_u}{4C'\chi_2}$, which simplifies to $\frac{1}{4C'}$ if we let $\chi_2=2\mu_{u}+\lambda_u$, implying that conditions (1) and (2) are satisfied.
\\

We now turn to prove that the Schur complement $\ZZ/\ZZ_{22}$ is positive semi-definite. After some straightforward calculations we have that
$$\ZZ_{12}^T\ZZ_{22}^{-1}\ZZ_{12}=\begin{bmatrix}
    \frac{\lambda_u^2}{2\mu_u+\lambda_u}(C'-\frac{1}{4\beta_u})(H_x\otimes I_y) && 0\\
    0 && \mu[C'-\frac{\mu_u}{4\beta_u(2\mu_u+\lambda_u)}](H_x\otimes I_y)
\end{bmatrix}.$$
Upon subtracting the above term from $\ZZ_{11}$ we note that the only remaining non-zero term of the complement is $[\ZZ/\ZZ_{22}]_{11}$. Hence, to show that $\ZZ/\ZZ_{22}$ is symmetric positive semi-definite, we need to make sure that the eigenvalues of the remaining term is non-negative. It should be noted that the term only consists of a constant which scales $H_x\otimes I_y$, thus it is sufficient to show that this constant is non-negative. Multiplying by $2\mu_u+\lambda_u$, the constant term satisfies
$$(2\mu_u+\lambda_u)^2C'-\frac{\lambda_u^2}{4\beta_u}-\lambda_u^2(C'-\frac{1}{4\beta_u})=4\mu_u(\mu_u+\lambda_u)C'> 0.$$
Note that the inequality follows from the assumption of \revone{the Lamé parameters being positive}, hence the Schur conditions are satisfied. 

\end{document}